\documentclass[reqno,10pt]{article}
\usepackage{etex}

\usepackage[style=alphabetic,backref,backend=bibtex,isbn=false,url=false,maxbibnames=99]{biblatex}
\renewbibmacro{in:}{}
\newbibmacro{string+doi}[1]{\iffieldundef{doi}{\iffieldundef{url}{#1}{\href{\thefield{url}}{#1}}}{\href{http://dx.doi.org/\thefield{doi}}{#1}}}
\DeclareFieldFormat*{title}{\usebibmacro{string+doi}{\emph{#1}}}
\DefineBibliographyStrings{english}{%
  backrefpage = {cited on p.},
  backrefpages = {cited on pp.},
}

\usepackage{amsmath,amscd,amsbsy,amsfonts,amsthm}
\usepackage{latexsym,amssymb,mathrsfs,bbm}
\usepackage{graphicx}
\usepackage{tikz}
\usetikzlibrary{arrows,calc,positioning,decorations.pathreplacing}
\usepackage[left=4cm,right=4cm]{geometry}
\usepackage{mhsetup}
\usepackage{mathtools}
\usepackage[bottom]{footmisc}
\renewcommand{\footnoterule}{%
  \kern -3pt
  \hrule width \textwidth height 0.4pt
  \kern 2.6pt
}

\usepackage[format=plain,indention=1em,labelsep=quad,labelfont={up},textfont={footnotesize},margin=2em]{caption}

\usepackage[colorlinks,citecolor=blue,linkcolor=blue,urlcolor=blue,filecolor=blue,breaklinks=true]{hyperref}
\usepackage{footnotebackref}

\makeatletter
\renewcommand*{\@seccntformat}[1]{\upshape \csname the#1\endcsname.\hspace{1ex}}
\renewcommand*{\section}{\@startsection{section}{1}{\z@}%
  {2.5ex \@plus 1ex \@minus 0.2ex}%
  {1.5ex \@plus 0.2ex}%
  {\normalfont\large\bfseries}}
\renewcommand*{\subsection}{\@startsection{subsection}{2}{\z@}%
  {2.5ex \@plus 1ex \@minus 0.2ex}%
  {1.5ex \@plus 0.2ex}%
  {\normalfont\normalsize\bfseries}}
\renewcommand*{\subsubsection}{\@startsection{subsubsection}{3}{\z@}%
  {2.5ex \@plus 1ex \@minus 0.2ex}%
  {-1.5ex \@plus -0.2ex}%
  {\normalfont\normalsize\bfseries}}
\renewcommand*{\paragraph}{\@startsection{paragraph}{4}{\z@}%
  {2.5ex \@plus 1ex \@minus 0.2ex}%
  {-1.5ex \@plus -0.2ex}%
  {\normalfont\normalsize\bfseries}}
\renewcommand*{\subparagraph}{\@startsection{subparagraph}{5}{\z@}%
  {2.5ex \@plus 1ex \@minus 0.2ex}%
  {-1.5ex \@plus -0.2ex}%
  {\normalfont\normalsize\slshape}}
\makeatother

\numberwithin{equation}{section}
\numberwithin{figure}{section}

\newcommand{\tel}{\ensuremath{\mathit{Tel}}}
\newcommand{\ibar}{\ensuremath{\bar{\imath}}}
\newcommand{\jbar}{\ensuremath{\bar{\jmath}}}
\newcommand{\signtwist}{^{(-1)}}
\newcommand{\emb}{\ensuremath{\hookrightarrow}}
\newcommand{\rmd}{\ensuremath{\mathrm{d}}}
\newcommand{\balltwo}{\ensuremath{\mathring{B}_2}}
\newcommand{\relC}{\ensuremath{\mathbf{C}}}
\newcommand{\relG}{\ensuremath{\mathbf{\Gamma}}}
\newcommand{\rels}{\ensuremath{\mathbf{s}}}
\newcommand{\reli}{\ensuremath{\mathbf{i}}}
\newcommand{\sfo}{\ensuremath{\mathsf{o}}}
\newcommand{\m}{\to}
\newcommand{\F}{\mathcal{F}}
\newcommand{\CC}{\mathfrak{C}}

\newcommand{\id}{\mathrm{id}}
\newcommand{\Emb}{\mathrm{Emb}}
\newcommand{\hocofib}{\mathit{hocofib}}
\newcommand{\hofib}{\mathit{hofib}}
\newcommand{\hocolim}{\mathit{hocolim}}
\newcommand{\colim}{\mathit{colim}}

\newcommand{\Aut}{\mathrm{Aut}}
\newcommand{\point}{\ensuremath{\text{\itshape pt}}}
\newcommand{\mbar}{\ensuremath{{\,\,\overline{\!\! M\!}\,}}}
\newcommand{\incl}[3][right]%
{%
\draw[<-,>=#1 hook] #2 to ($ #2!0.5!#3 $);
\draw[->] ($ #2!0.5!#3 $) to #3;%
}
\newcommand{\inclusion}[5][right]%
{%
\draw[<-,>=#1 hook] #4 to ($ #4!0.5!#5 $) node[#2,font=\small]{#3};
\draw[->] ($ #4!0.5!#5 $) to #5;%
}

\renewcommand{\leq}{\leqslant}
\renewcommand{\geq}{\geqslant}

\newcommand{\cB}{\mathcal{B}}

\newcommand{\cE}{\mathcal{E}}
\newcommand{\cF}{\mathcal{F}}

\newcommand{\cM}{\mathcal{M}}
\newcommand{\cN}{\mathcal{N}}
\newcommand{\cO}{\mathcal{O}}

\newcommand{\bN}{\mathbb{N}}

\newcommand{\bR}{\mathbb{R}}

\newcommand{\bZ}{\mathbb{Z}}

\newenvironment{itemize2}%
{\begin{itemize}

\setlength{\itemsep}{1pt}
\setlength{\parskip}{0pt}
\setlength{\parsep}{0pt}}%
{\end{itemize}}

\theoremstyle{plain}
\newtheorem{theorem}{Theorem}[section]
\newtheorem{lemma}[theorem]{Lemma}

\newtheorem{proposition}[theorem]{Proposition}
\newtheorem{corollary}[theorem]{Corollary}
\newtheorem{atheorem}{Theorem}

\theoremstyle{definition}
\newtheorem{definition}[theorem]{Definition}
\newtheorem{remark}[theorem]{Remark}

\begin{document}

\title{\Large\scshape Scanning for oriented configuration spaces}
\author{\small Jeremy Miller and Martin Palmer}
\date{\small\today}
\maketitle
{
\makeatletter
\renewcommand*{\BHFN@OldMakefntext}{}
\makeatother
\footnotetext{During part of the preparation of this article the second author was supported by Michael Weiss' Alexander von Humboldt Professorship grant.}
}

\begin{abstract}
\noindent In \cite{Palmer2013} the second author proved that the sequence of `oriented' configuration spaces on an open connected manifold exhibits homological stability as the number of particles goes to infinity. To complement that result we identify the corresponding limiting space, up to homology equivalence, as a certain explicit double cover of a section space. Along the way we also prove that the scanning map of \cite{Mc1} for unordered configuration spaces is acyclic in the limit.
\end{abstract}

\section{Introduction}

There are many interesting examples of families of spaces $\{Y_k\}$ whose homology groups $H_i(Y_k)$ are independent of $k$ for $k\gg i$. Examples include the classifying spaces of general linear groups \cite{Q,Charney1980}, mapping class groups \cite{Ha,Wahl2008}, automorphism groups of free groups \cite{Hatcher1995,HatcherVogtmann1998,HatcherWahl2010} and unordered configuration spaces of particles in an open connected manifold \cite{Mc1,Se}. In many of these cases (see for example \cite{Mc1,MW,Galatius2011}) one can find a computationally more tractable space $Z$ which is homology equivalent to the limiting space $\hocolim_k (Y_k)$. In \cite{Palmer2013} the second author proved homological stability for \emph{oriented configuration spaces} and the purpose of this paper is to describe the corresponding limiting space.

Oriented configuration spaces are natural generalizations of the classifying spaces of the alternating groups. One possible motivation for their study was given in \cite{GKY}, where it was shown that homological stability for oriented configuration spaces implies stability for the homotopy groups of spaces of positive and negative particles. We will also describe an application of these ideas to the study of the homology of the spaces appearing in the generalized Snaith splitting of \cite{B}.

\paragraph{Background.}
Before we state the results of this paper and of \cite{Palmer2013}, we first fix some notation and review some classical theorems regarding configuration spaces of unordered particles. Let $F_k(M)\coloneqq M^k \smallsetminus\Delta_f$ where $\Delta_f$ is the fat diagonal. Define $C_k(M)$ to be the quotient of $F_k(M)$ by the action of the symmetric group $\Sigma_k$, and $C^+_k(M)$ to be the quotient by the action of the alternating group $A_k$. We call these spaces respectively the configuration spaces of ordered, unordered and oriented collections of points in $M$.

Throughout, we require that the manifold $M$ be connected and of dimension at least $2$. We say that a manifold \emph{admits boundary} if it is the interior of a (not necessarily compact) manifold with (not necessarily compact) non-empty boundary. For such manifolds, Segal proved in \cite{Se} the following theorem.

\begin{theorem}[{\cite[Proposition A.1]{Se}}]\label{unorStab}
If $M$ is a manifold admitting boundary there is a map $t\colon C_k(M) \m C_{k+1}(M)$ which induces an isomorphism on homology for $* \leq k/2$.
\end{theorem}

We call the map $t$ the ``stabilization map.'' Roughly, it involves moving all the particles away from the boundary and then adding a new particle near the boundary; see \S\ref{ss-concrete-models} for precise definitions. No such map exists for closed manifolds and in fact homological stability fails for closed manifolds in general.\footnote{For example, from the presentation of $\pi_1(C_k(S^2))$ in \cite{FV} we have $H_1(C_k(S^2);\bZ) \cong \bZ/(2k-2)$, which is not stable as $k\to\infty$. Homological stability does, however, hold rationally \cite{Ch,Randal-Williams2013} and for mod-$2$ coefficients \cite{MilgramLoeffler1988,BCT,Randal-Williams2013}. See also \cite{BenderskyMiller2014} and \cite{CanteroPalmer2014} for more stability results for the torsion in the homology of unordered configuration spaces on closed manifolds.}

Let $\pi\colon\dot TM\m M$ denote the fiberwise one-point compactification of the tangent bundle of $M$ and let $\Gamma(M)$ denote the space of compactly-supported sections of this bundle. There is a natural bijection $\pi_0(\Gamma(M)) \to \bZ$ given by the \emph{degree} of a compactly-supported section, whose definition we recall in \S\ref{sssDegree}. For $k\in\bZ$ we denote the path-component of $\Gamma(M)$ consisting of degree-$k$ sections by $\Gamma_k(M)$. In \cite{Mc1}, {McDuff} defined a scanning map $s\colon C_k(M) \m \Gamma_k(M)$ and proved the following two theorems.

\begin{theorem}[{\cite[Theorem 1.2]{Mc1}}]\label{scanlim}
If $M$ is a manifold admitting boundary the scanning maps $s\colon C_k(M) \m \Gamma_k(M)$ induce an isomorphism $H_*(C_\infty(M);\bZ) \to H_*(\Gamma_\infty(M);\bZ)$.
\end{theorem}

Here $C_\infty(M)$ denotes the homotopy colimit of the maps
\[
\cdots \to C_k(M) \to C_{k+1}(M) \to \cdots
\]
from Theorem \ref{unorStab}, and $\Gamma_\infty(M)$ denotes the homotopy colimit of analogous ``stabilization'' maps for the path-components $\Gamma_k(M)$ of $\Gamma(M)$.

\begin{theorem}[{\cite[Theorem 1.1]{Mc1}}]\label{scanrange}
The scanning map $s\colon C_k(M) \m \Gamma_k(M)$ induces an isomorphism on homology in the same range \textup{(}$* \leq k/2$\textup{)} as the map $t\colon C_k(M \smallsetminus \point) \m C_{k+1}(M \smallsetminus \point)$.
\end{theorem}

When the manifold $M$ admits boundary the spaces $\Gamma_k(M)$ are all homotopy equivalent (see \S\ref{sssStabScanMaps}), so this provides a limiting space $Z$ for the sequence $\{C_k(M)\}$. When $M$ is closed, we just have two sequences $\{C_k(M)\}$ and $\{\Gamma_k(M)\}$ which become better and better approximations of each other as $k\to\infty$.

\paragraph{Oriented configuration spaces.}

The analogue of Theorem \ref{unorStab} for oriented configuration spaces is the following:
\begin{theorem}[{\cite{Palmer2013}}]\label{altstab}
Let $M$ be a manifold admitting boundary. Then there is a map $t^+\colon C^+_k(M) \m C^+_{k+1}(M)$ which induces an isomorphism on homology for $* \leq (k-5)/3$ and a surjection for $* \leq (k-2)/3$.
\end{theorem}

The goal of this paper is to provide analogues of Theorem \ref{scanlim} and Theorem \ref{scanrange} for oriented configuration spaces. For $k\geq 2$, the scanning map $s\colon C_k(M) \m \Gamma_k(M)$ induces an isomorphism on $H^1(-;\bZ/2)$, by Theorem \ref{scanrange} above and the universal coefficient theorem. Cohomology with mod-$2$ coefficients classifies double covers of a space, so this says that any double cover of $C_k(M)$ is the pullback of a (unique) double cover of $\Gamma_k(M)$. In particular $C_k^+(M)$ fits into a pullback square:
\begin{equation}\label{ePullback}
\centering
\begin{split}
\begin{tikzpicture}
[x=1.5mm,y=1.2mm]
\node (tl) at (0,10) {$C^+_k(M)$};
\node (tr) at (20,10) {$\Gamma^+_k(M)$};
\node (bl) at (0,0) {$C_k(M)$};
\node (br) at (20,0) {$\Gamma_k(M).$};
\node at (3,6.5) {$\lrcorner$};
\draw[->] (tl) to node[above,font=\small]{$s^+$} (tr);
\draw[->] (bl) to node[above,font=\small]{$s$} (br);
\draw[->] (tl) to (bl);
\draw[->] (tr) to (br);
\end{tikzpicture}
\end{split}
\end{equation}
There is an alternative, more geometric description of the associated double cover $\Gamma^+_k(M) \to \Gamma_k(M)$ which is described in \S\ref{sssDoubleCover}. Our analogue of Theorem \ref{scanrange} is:

\begin{atheorem}\label{orientrange}
The lifted scanning map $s^+\colon C^+_k(M) \m \Gamma^+_k(M)$ induces an isomorphism on homology in the range $*\leq (k-5)/3$ and a surjection for $*\leq (k-2)/3$.
\end{atheorem}

Again, when the manifold $M$ admits boundary the spaces $\Gamma_k^+(M)$ are all homotopy equivalent, so this provides a limiting space $Z$ for the sequence $\{C_n^+(M)\}$.

Note that unlike configuration spaces of unordered particles, or more generally configuration spaces with summable labels \cite{Sa}, oriented configuration spaces are not local: to determine a point in $C^+_k(M)$ one needs more than the information attached to each point in the configuration. Nevertheless, they still exhibit homological stability and we can still, via Theorem \ref{orientrange}, describe a limiting space.

To prove Theorem \ref{orientrange} we first prove the following strengthening of Theorem \ref{scanlim}:

\begin{atheorem}\label{acyclic}
If $M$ is a manifold admitting boundary, the scanning map in the limit $s\colon C_\infty(M)\to \Gamma_\infty(M)$ is acyclic.
\end{atheorem}

In \S\ref{ss-concrete-models} we define concrete models for configuration and section spaces so that the stabilization and scanning maps commute on the nose. The scanning maps therefore induce a well-defined map from the mapping telescope of the sequence of stabilization maps for configuration spaces to the mapping telescope of the sequence of stabilization maps for section spaces. This is the map $s$ referred to in the above theorem.

This theorem, combined with the stability result from \cite{Palmer2013}, directly implies Theorem \ref{orientrange} in the case of manifolds admitting boundary (see Corollary \ref{rangeopen}). In \S\ref{ss-closed} we show how to extend Theorem \ref{orientrange} to closed manifolds.

We emphasize that our method of deducing Theorem \ref{orientrange} from Theorem \ref{acyclic} uses homological stability for oriented configuration spaces with untwisted coefficients, so we cannot deduce that the lifted scanning map $s^+\colon C_k^+(M) \to \Gamma_k^+(M)$ is \emph{acyclic} in a range of degrees depending on $k$. We only know that it is a homology equivalence in a range, and acyclic in all homological degrees when $k\to\infty$. The question of acyclicity in a range before taking the limit is not pursued here.

To prove Theorem \ref{acyclic} we need the notion of an \emph{abelian homology fibration}, a twisted version of the homology fibration criterion of \cite{Mc1} and a corollary of a twisted version of the group-completion theorem. The full details of all this are worked out in \cite{MillerPalmer2014}, following the ideas of \cite{Mc1} and \cite{MSe} and using a result of \cite{Randal-Williams2013Groupcompletion}. The particular definitions and results which are needed for the present paper are recalled in \S\ref{s-prelim}.

\paragraph{Twisted and abelian homology equivalences.}

We now fix terminology for some different notions of homology equivalence. Let $f\colon X\to Y$ be a continuous map of spaces, and let $\cF$ denote a local coefficient system on $Y$ -- this can be thought of as a functor $\pi(Y)\to\mathsf{Ab}$ from the fundamental groupoid of $Y$ to the category of abelian groups or as a bundle of abelian groups over $Y$. It is called \emph{abelian} if, in the bundle viewpoint, the monodromy of any fiber around a commutator loop is trivial. In the functor viewpoint this says that for each $y\in Y$ the homomorphism $\pi_1(Y,y)=\pi(Y)(y,y)\to \Aut_{\mathsf{Ab}}(\cF(y))$ factors through an abelian group.

The map $f$ is called \emph{acyclic}, or a \emph{twisted homology equivalence}, if it induces isomorphisms $H_*(X;f^*\cF)\to H_*(Y;\cF)$ for all local coefficient systems $\cF$ on $Y$. It is called an \emph{abelian homology equivalence} if it induces isomorphisms for all abelian local coefficient systems on $Y$, and it is called a \emph{trivial homology equivalence}, or just a \emph{homology equivalence}, if it induces isomorphisms for the trivial coefficient system $\bZ$ (i.e.\ the trivial bundle $\bZ\times Y\to Y$). An alternative characterization (see \cite[Proposition 4.3]{Berrick1982}) of acyclicity of $f$ is that $\widetilde{H}_*(\hofib(f);\bZ)=0$ in all degrees (where $\bZ$ is a trivial coefficient system). From this it follows that the homotopy pullback of an acyclic map is acyclic. In particular in diagram \eqref{ePullback}, once $k\to\infty$, acyclicity of $s$ will imply acyclicity of $s^+$.

\paragraph{The sign representation.}

One can rephrase results about oriented configuration spaces in terms of homology of unordered configuration spaces with certain twisted coefficients. Let $\rho\colon \pi_1(C_k(M)) \m \bZ/2$ be the composition of the natural map $\pi_1(C_k(M)) \m \Sigma_k$ and the sign homomorphism $\Sigma_k \m \bZ/2$. For a ring $R$, the group-ring $R[\bZ/2]$ is given the structure of an $R[\pi_1(C_k(M))]$-module by the homomorphism $\rho$. By the definition of local homology, or a trivial application of the Serre spectral sequence to the fibration $S^0 \m C^+_k(M) \m C_k(M)$, we have an isomorphism:
\[
H_*(C^+_k(M);R) \;\cong\; H_*(C_k(M);R[\bZ/2]).
\]
When $2$ is invertible in $R$, the module $R[\bZ/2]$ decomposes as $R\oplus R\signtwist$, where $\pi_1(C_k(M))$ acts trivially on $R$, and on $R\signtwist$ it acts by $\rho$ and the action of $\bZ/2$ given by $r\mapsto -r$ (the ``sign representation''). Hence we have a decomposition:
\[
H_*(C^+_k(M);R) \;\cong\; H_*(C_k(M);R) \oplus H_*(C_k(M);R\signtwist).
\]

Theorem \ref{acyclic} allows one to study the homology of the spaces $C_k(M)$ with this twisted coefficient system. The groups $H_*(C_k(M);\bZ\signtwist)$ are interesting for the following reason. Let $M$ be a closed, almost parallelizable $d$-manifold. Then for any $m>0$ the space $\mathit{Map}_*(M,S^{d+m})$ of based maps splits stably (in the sense of stable homotopy theory) into summands which are Thom spaces of vector bundles over $C_k(M \smallsetminus \point)$ (Propositions 2 and 3 of \cite{B}; see also \cite{BCT}). The construction recovers the Snaith splitting \cite{Sn} when $M=S^d$. The Thom isomorphism theorem implies that the homologies of these Thom spaces are shifts of $H_*(C_k(M\smallsetminus\point);\bZ)$ or $H_*(C_k(M\smallsetminus\point);\bZ\signtwist)$ depending on whether or not the corresponding vector bundles are orientable. So to understand the homology of the spaces appearing in generalized Snaith splitting one needs to understand the homology of configuration spaces with sign-twisted coefficients.

\paragraph{Outline.}

We begin in \S\ref{s-prelim} by recalling two important tools which we will need: we recall from \cite{MillerPalmer2014} a twisted version of McDuff's homology fibration criterion (Proposition \ref{strongcrit}) and a corollary of a twisted version of the group-completion theorem (Proposition \ref{cGroupCompletion2}). In \S\ref{ss-concrete-models} we then carefully describe models for the stabilization and scanning maps which commute on the nose. In \S\ref{ss-r2crossn} we prove acyclicity of the scanning map in the limit (Theorem \ref{acyclic}) in the special case when the manifold $M$ is $\bR^m$. For this we need Proposition \ref{cGroupCompletion2} applied to a certain topological monoid built from configurations in $\bR^m$. In \S\ref{ss-open} this special case is used to extend Theorem \ref{acyclic} to any manifold $M$ admitting boundary; this step requires the use of the abelian homology fibration criterion (Proposition \ref{strongcrit}). From this we deduce that the lifted scanning map is a homology equivalence in a stable range (Theorem \ref{orientrange}) when the manifold $M$ admits boundary. Finally, in \S\ref{ss-closed} we show how to extend this latter result to closed manifolds and also discuss the scanning map on homology twisted by the sign representation.

\paragraph{Acknowledgments.}

We would like to thank Oscar Randal-Williams and Ulrike Tillmann for several enlightening discussions. Additionally, we thank the anonymous referee for many helpful suggestions and corrections, which have much improved this paper.

\section{An abelian homology fibration criterion and a corollary of the group-completion theorem } \label{s-prelim}

In this section we review some results from \cite{MSe}, \cite{Randal-Williams2013} and \cite{MillerPalmer2014}. We recall the definition of an abelian homology fibration, give a criterion for recognizing these and then state a corollary of the (twisted) group-completion theorem. In \cite{MillerPalmer2014} one works with a class of local coefficient systems $\CC$ which is closed under taking pullbacks. Here we specialize $\CC$ to the class of all abelian local coefficient systems, but the results in \cite{MillerPalmer2014} apply in the general situation.

Given a map $r\colon Y \to X$ and a point $z\in X$, we denote the homotopy fiber of $r$ over $z$ by $\hofib_z(r)$.

\begin{definition}\label{dAbelianHF}
A map $r\colon Y \m X$ is called an \emph{abelian homology fibration} if for all points $z \in X$ the natural inclusion $r^{-1}(z) \m  \hofib_z(r)$ is an abelian homology equivalence. That is, if $\F$ is an abelian local coefficient system on $\hofib_z(r)$ and $i\colon r^{-1}(z) \m \hofib_z(r)$ is the natural inclusion, then $i$ induces an isomorphism:
\[
i_*\colon H_*(r^{-1}(z);i^*\F) \longrightarrow H_*(\hofib_z(r);\F).
\]
\end{definition}

To state the recognition criterion for abelian homology fibrations we need the following:

\begin{definition}\label{dLocallyStalklike}
We say that a map $r\colon Y\to X$ is \emph{locally stalk-like} over $Z\subseteq X$ if there is a basis $\cB$ for the topology of $Z$ such that each $U\in \cB$ is contractible and contains a point $z_U$ such that the inclusion $r^{-1}(z_U) \hookrightarrow r^{-1}(U)$ is a weak equivalence. 
\end{definition}

Note that this is a much weaker condition than requiring this inclusion to be a weak equivalence for \emph{every} point in $U$.

The following is a version of McDuff's homology fibration criterion (Proposition 5.1 of \cite{Mc1}) which applies to the class of abelian local coefficient systems. It follows from Theorem 3.1 and Remark 3.2 of \cite{MillerPalmer2014} with $\CC$ taken to be the class of all abelian local coefficient systems.

\begin{proposition}\label{strongcrit}
Let $X$ be a topological space with closed filtration $\{X_n\}_{n\in\bN}$, meaning that the $X_n$ are closed subsets of $X$ satisfying $X_{n-1}\subseteq X_n$, $X=\bigcup_{n\in\bN} X_n$ and each compact subset of $X$ is contained in some $X_n$. Let $r\colon Y \to X$ be a map and assume that each $X_n$ and $r^{-1}(X_n)$ are Hausdorff. Then $r$ is an abelian homology fibration if the following three conditions are satisfied\textup{:}
\begin{itemize2}
\item[\textup{(i)}] The map $r$ is locally stalk-like over each $X_n$.
\item[\textup{(ii)}] The restriction of $r$ to each $X_n \smallsetminus X_{n-1}$ and to $X_0$ is an abelian homology fibration.
\item[\textup{(iii)}] For each $n$ there is $X_{n-1}\subseteq U_n\subseteq X_n$, with $U_n$ open in $X_n$, and homotopies $h_t \colon U_n \m U_n$ and $H_t \colon r^{-1}(U_n) \m r^{-1}(U_n)$ satisfying\textup{:}
  \begin{itemize2}
  \item[\textup{(a)}] $h_0=id$, $h_t(X_{n-1}) \subseteq X_{n-1}$, $h_1(U_n) \subseteq X_{n-1}$\textup{;}
  \item[\textup{(b)}] $H_0 = id$, $r \circ H_t = h_t \circ r$\textup{;}
  \item[\textup{(c)}] for all $x \in U_n$, $H_1\colon r^{-1}(x) \m r^{-1}(h_1(x))$ is an abelian homology equivalence.
  \end{itemize2}
\end{itemize2}
\end{proposition}

We now discuss an application of the group-completion theorem to topological monoids. Let $\cM$ be a topological monoid with $\pi_0(\cM)=\bN$. Denote its components by $\cM_k$, choose an element $m\in \cM_1$ and define $\cM_\infty$ to be the mapping telescope of the sequence of maps $\cM\to \cM\to \cM\to\cdots$, where each map is right-multiplication by $m$.

Remark 2 of \cite{MSe} states that there is a weak equivalence between $\cM_{\infty}^+$ and $\Omega B\cM$ when $\cM$ is homotopy-commutative. Here the notation $X^+$ denotes the Quillen plus construction with respect to the maximal perfect subgroup of $\pi_1(X)$,\footnote{In this case it follows from the fact that $\pi_1(\cM_{\infty}^+) \cong \pi_1(\Omega B\cM)$ is abelian that the maximal perfect subgroup of $\pi_1(\cM_\infty)$ is its commutator subgroup. This can also be proved directly; see Proposition 3.1 of \cite{Randal-Williams2013Groupcompletion}.} $\Omega$ is the based loop space functor and $B$ is the classifying space functor. However, some details are missing from \cite{MSe} to extract a complete proof of this fact. One needs to know first that the action of $\cM$ on $\cM_\infty$ induced by left-multiplication is by abelian homology equivalences when $\cM$ is homotopy-commutative. A detailed proof of this has been written in \cite{Randal-Williams2013Groupcompletion}. Second, one needs to know that McDuff and Segal's group-completion theorem (Proposition 2 of \cite{MSe}) is also true when ``homology equivalence'' and ``homology fibration'' are replaced by ``abelian homology equivalence'' and ``abelian homology fibration''. This twisted version of the group-completion theorem is proved as Theorem 4.1 of \cite{MillerPalmer2014}.\footnote{More generally, Theorem 4.1 of \cite{MillerPalmer2014} proves a version of the group completion theorem where ``homology equivalence'' and ``homology fibration'' are replaced by ``$\CC$-homology equivalence'' and ``$\CC$-homology fibration'' for any class $\CC$ of twisted coefficient systems which is closed under taking pullbacks. In contrast, the result of \cite{Randal-Williams2013Groupcompletion} that $\cM$ acts on $\cM_\infty$ by abelian homology equivalences is sharp: it is in general not true that $\cM$ acts on $\cM_\infty$ by twisted homology equivalences, i.e.\ acyclic maps.}

Putting this together, we indeed have a weak equivalence $\cM_{\infty}^+ \simeq \Omega B\cM$ (c.f.\ Corollary 1.2 of \cite{Randal-Williams2013Groupcompletion} and Corollary 4.2 of \cite{MillerPalmer2014}). Since loop spaces are simple we deduce:

\begin{proposition}\label{cGroupCompletion2}
For a homotopy-commutative monoid $\cM$ the space $\cM_{\infty}^+$ is simple.
\end{proposition}

See \cite{Randal-Williams2013Groupcompletion} and \S 4 of \cite{MillerPalmer2014} for a more detailed discussion of this result.


\section{Models for scanning and stabilization} \label{ss-concrete-models}

In this section we define concrete models for all our configuration and section spaces, and define the stabilization and scanning maps in these models so that they commute on the nose. We also give the more geometric description of the covering space $\Gamma^+(M) \to \Gamma(M)$ promised in the introduction.

Fix a connected, smooth, $m$-dimensional manifold $M$. As mentioned in the introduction, the unordered configuration space $C_k(M)$ of $k$ points in $M$ is defined to be the quotient of $M^k\smallsetminus \Delta_f$ by the action of the symmetric group $\Sigma_k$, where $\Delta_f$ is the fat diagonal $\{ (p_1,\dotsc,p_k) \mid p_i=p_j \text{ for some } i\neq j \}$. The space $\Gamma(M)$ is defined to be the space of compactly-supported sections of $\dot{T}M$, the fiberwise one-point compactification of the tangent bundle of $M$, given the subspace topology from the compact-open topology on $\mathrm{Map}(M,\dot{T}M)$. In \S\ref{sssDegree} we define the \emph{degree} of such a section; we denote by $\Gamma_k(M)$ the subspace (in fact, path-component) of $\Gamma(M)$ of sections of degree $k\in\bZ$.

We then define the models for the scanning and stabilization maps to be used in \S\S\ref{ss-r2crossn}--\ref{ss-closed}. For this some auxiliary data is needed. Given a Riemannian metric $g$ on $M$ and a smooth function $\rho\colon M\to (0,1)$ which is a strict lower bound for the injectivity radius of $(M,g)$ we define a scanning map
\begin{equation}\label{eScanning}
s = s(g,\rho) \colon C_k(M) \longrightarrow \Gamma_k(M).
\end{equation}

To define the stabilization maps we first define alternative models for the configuration and section spaces, depending on the following auxiliary data. First, we choose an embedding $e\colon M\emb (-1,1)^d\times \bR$ for some $d$ such that
\begin{equation}\label{eEmbeddingCondition}
e(M)\cap \pi_{d+1}^{-1}((-3,\infty)) = (-1,1)^{m-1}\times (-3,\infty),
\end{equation}
where $\pi_{d+1}$ is the projection onto the $(d+1)$st coordinate. Such an embedding exists if and only if $M$ is non-compact. This in particular chooses an ``end'' of $M$, namely a proper homotopy class of properly embedded rays in $M$. We will also refer to $(-1,1)^{m-1}\times (-3,\infty)$ or its preimage under $e$ as the \emph{chosen end of $M$}. Second, we choose a Riemannian metric $g$ and function $\rho$ as above such that the ball around $p$ in $M$ of radius $\rho(p)$ does not extend too far in the $(d+1)$st coordinate (see \S\ref{sssStabScanMaps} for the precise condition). The Riemannian metric is not required to be the one inherited from $(-1,1)^d \times \bR$ via the embedding $e$; the purpose of the embedding is simply to (a) choose an end of the manifold and (b) provide convenient coordinates for that end.

Given such data $(e,g,\rho)$ we define in \S\ref{sssConfigSectionSpaces} spaces $C_k(M,e)$ and $\Gamma_k(M,e)$ and construct homotopy equivalences $C_k(M) \simeq C_k(M,e)$ and $\Gamma_k(M) \simeq \Gamma_k(M,e)$. In \S\ref{sssStabScanMaps} we construct stabilization maps $t\colon C_k(M,e) \to C_{k+1}(M,e)$ and $T\colon \Gamma_k(M,e) \to \Gamma_{k+1}(M,e)$ and also a scanning map $s\colon C_k(M,e) \to \Gamma_k(M,e)$. The square
\begin{equation}\label{eScanStabSquare}
\centering
\begin{split}
\begin{tikzpicture}
[x=1mm,y=1mm]
\node (tl) at (0,15) {$C_k(M,e)$};
\node (tr) at (30,15) {$\Gamma_k(M,e)$};
\node (bl) at (0,0) {$C_{k+1}(M,e)$};
\node (br) at (30,0) {$\Gamma_{k+1}(M,e)$};
\draw[->] (tl) to node[above,font=\small]{$s$} (tr);
\draw[->] (bl) to node[below,font=\small]{$s$} (br);
\draw[->] (tl) to node[left,font=\small]{$t$} (bl);
\draw[->] (tr) to node[right,font=\small]{$T$} (br);
\end{tikzpicture}
\end{split}
\end{equation}
commutes on the nose (see \S\ref{sssCommutativity}) and the square
\begin{equation}\label{eTwoScanningMaps}
\centering
\begin{split}
\begin{tikzpicture}
[x=1mm,y=1mm]
\node (tl) at (0,15) {$C_k(M)$};
\node (tr) at (30,15) {$\Gamma_k(M)$};
\node (bl) at (0,0) {$C_k(M,e)$};
\node (br) at (30,0) {$\Gamma_k(M,e)$};
\draw[->] (tl) to node[above,font=\small]{$s$} (tr);
\draw[->] (bl) to node[below,font=\small]{$s$} (br);
\draw[->] (tl) to node[left,font=\small]{$\simeq$} (bl);
\draw[->] (tr) to node[right,font=\small]{$\simeq$} (br);
\end{tikzpicture}
\end{split}
\end{equation}
(where the top map is \eqref{eScanning}) commutes up to a homotopy described in \S\ref{sssStabScanMaps}. In \S\ref{sssOriented} we explain the modifications needed to define the stabilization map for oriented configuration spaces.

In \S\ref{sssDoubleCover} we geometrically construct a homomorphism $\pi_1(\Gamma_k(M,e)) \to \bZ/2$. This determines double covers (up to isomorphism) of all the spaces in the square \eqref{eTwoScanningMaps} above, and the induced double cover of $C_k(M)$ is the oriented configuration space. Hence this definition of $\Gamma_k^+(M)$ agrees with the introduction, where it was characterized as the unique double cover which pulls back along the scanning map to the oriented configuration space.

\begin{remark}
In \S\S\ref{ss-r2crossn} and \ref{ss-open} we use the models $C_k(M,e)$, $\Gamma_k(M,e)$ etc., whereas in \S\ref{ss-closed} we use $C_k(M)$, $\Gamma_k(M)$ and the model \eqref{eScanning} of the scanning map. If we choose a particular double cover $\Gamma_k^+(M,e)$ in the isomorphism class that we have defined and pull this back in the square \eqref{eTwoScanningMaps} to get particular double covers of the other spaces, then we get two double covers of $C_k(M)$ and the homotopy filling \eqref{eTwoScanningMaps} determines an isomorphism between them. Temporarily denoting these two double covers by $C_k^{\top}(M)$ and $C_k^{\perp}(M)$, the homotopy in \eqref{eTwoScanningMaps} also determines a homotopy filling the square:
\begin{equation}\label{eTwoScanningMapsLifted}
\centering
\begin{split}
\begin{tikzpicture}
[x=1mm,y=1mm]
\node (tl1) at (0,15) {$C_k^{\top}(M)$};
\node (tl2) at (20,15) {$C_k^{\perp}(M)$};
\node (tr) at (50,15) {$\Gamma_k^+(M)$};
\node (bl) at (0,0) {$C_k^+(M,e)$};
\node (br) at (50,0) {$\Gamma_k^+(M,e).$};
\draw[->] (tl2) to node[above,font=\small]{$s^+$} (tr);
\draw[->] (bl) to node[below,font=\small]{$s^+$} (br);
\draw[->] (tl1) to node[left,font=\small]{$\simeq$} (bl);
\draw[->] (tr) to node[right,font=\small]{$\simeq$} (br);
\node at (10,15) {$\cong$};
\end{tikzpicture}
\end{split}
\end{equation}
Hence for the purposes of \S\ref{ss-closed} (showing that the lift $s^+$ of the scanning map to the double covers is an isomorphism on homology in a range) the two models for $s^+$ are equivalent.
\end{remark}

\subsection{Degree of a section.}\label{sssDegree}
We begin by defining the \emph{degree} of a compactly supported section of $\pi\colon \dot TM\to M$. Let $\cO_M$ denote the orientation bundle of $M$, in other words the bundle of abelian groups whose fiber over $m\in M$ is $H_d(M,M\smallsetminus\{m\};\bZ)$. Similarly, let $\cO_{\dot TM}$ denote the orientation bundle of the manifold $\dot TM$, i.e.\ the bundle whose fiber over $\widetilde{m}\in \dot TM$ is $H_{2d}(\dot TM,\dot TM\smallsetminus\{\widetilde{m}\} ; \bZ)$.

Let $\sigma$ be a section of $\pi\colon \dot TM\to M$. We then have two bundles over $M$, namely $\cO_M$ and $\sigma^* \cO_{\dot TM}$, which we claim are canonically isomorphic: Over $m\in M$ their fibers are $H_d(M,M\smallsetminus\{m\};\bZ)$ and $H_{2d}(\dot TM,\dot TM\smallsetminus\{\sigma(m)\} ; \bZ)$ respectively. These are of course abstractly isomorphic, but moreover excision and the suspension isomorphism induce canonical isomorphisms for each $m$ which assemble to form a bundle isomorphism $\cO_M \cong \sigma^* \cO_{\dot TM}$.

Now we have a class $u\in H_c^d(\dot TM;\cO_{\dot TM})$ which is the Poincar{\'e} dual to the zero section $M\subset \dot TM$. Pulling this back along the induced map $\sigma^*$ on compactly-supported cohomology gives a class $\sigma^* u\in H_c^d(M;\sigma^* \cO_{\dot TM}) \cong H_c^d(M;\cO_M)$. Applying the Poincar{\'e} duality isomorphism we therefore have a class in $H_0(M;\bZ)=\bZ$. The \emph{degree} of $\sigma$ is defined to be this integer. Given two sections $\sigma, \tau \in \Gamma(M)$ in the same path-component, the maps $\sigma^*$ and $\tau^*$ on compactly-supported cohomology are equal, so $\mathrm{deg}(\sigma)=\mathrm{deg}(\tau)$. So this defines a map $\pi_0(\Gamma(M)) \to \bZ$.

Denote the subspace of $\Gamma(M)$ of degree-$k$ sections by $\Gamma_k(M)$; by the previous paragraph this is a union of path-components. Fix an open ball $B\subseteq M$ and note that the space of sections of $\pi$ with compact support in $B$ is homotopy equivalent to $\Omega^d S^d$. So clearly each $\Gamma_k(M)$ is non-empty since one can construct sections of any degree supported in $B$. Moreover, one can see that each $\Gamma_k(M)$ is path-connected as follows. First, as pointed out in \cite[Remark 3.8.8]{LurieDerivedAlgebraicGeometry}, each compactly-supported section of $\pi$ is homotopic (through such sections) to one whose support is contained in $B$.\footnote{The argument is as follows: let $M_{d-1}\subseteq M$ be the $(d-1)$-skeleton of a triangulation of $M$ (which exists since $M$ is smooth). The bundle $\pi|_{M_{d-1}}$ has fibers whose connectivity is at least the dimension of its base, so by obstruction theory all sections of $\pi|_{M_{d-1}}$ are homotopic (through sections). Therefore any section of $\pi$ may be homotoped to be supported in $M\smallsetminus M_{d-1}$, which is a union of disjoint open balls. Since the section was compactly-supported to begin with, this new section will have support contained in a finite union of these open balls, and therefore contained in some (maybe larger) open ball in $M$. This can then be homotoped to be supported inside the given ball $B$.} Given two compactly-supported sections $\sigma$ and $\tau$ of degree $k$, one can homotope them to sections $\sigma^\prime$ and $\tau^\prime$ with compact support in $B$, which will again have degree $k$. The fact that the space $\Omega^d S^d$ has $\pi_0$ equal to $\bZ$ (identified by the degree) means that one can homotope $\sigma^\prime$ to $\tau^\prime$ through sections supported in $B$.

Note that one could also \emph{define} the degree of $\sigma\in\Gamma(M)$ using the fact mentioned in the previous paragraph: first homotope $\sigma$ to a section $\sigma^\prime$ whose support is contained in $B$, so we can think of $\sigma^\prime$ as an element of $\Omega^d S^d$, and then define $\deg(\sigma) \coloneqq \deg(\sigma^\prime)$. A priori this is not necessarily well-defined, but in fact it is since it agrees with the definition given above.

\subsection{Configuration and section spaces.}\label{sssConfigSectionSpaces}

We now define alternative models for configuration and section spaces, which will allow us to define stabilization and scanning maps which commute on the nose. Recall that we defined $C_k(M)$ as the quotient space $(M^k\smallsetminus \Delta_f)/\Sigma_k$, where $\Delta_f$ is the fat diagonal in $M^k$, and $\Gamma(M)$ as the space of compactly-supported sections of $\dot TM\to M$. We write $\Gamma_k(M)$ for the path-component of $\Gamma(M)$ consisting of sections of degree $k$ and $C(M)$ for the disjoint union of $C_k(M)$ over all integers $k\geq 0$.

Assume that we have an embedding $e\colon M\emb (-1,1)^d \times \bR$ satisfying \eqref{eEmbeddingCondition}. Write $\hat{M}=e(M)$ and for $t\in [-3,\infty)$ define $\hat{M}_t \coloneqq e(M)\cap \pi_{d+1}^{-1}((-\infty,t))$. Then we define:
\begin{align*}
C_k(M,e) &\coloneqq \{ (c,t)\in C_k(\hat{M})\times [0,\infty) \mid c\subseteq \hat{M}_t \} \\
\Gamma_k(M,e) &\coloneqq \{ (\sigma,t)\in \Gamma_k(\hat{M})\times [0,\infty) \mid \mathrm{supp}(\sigma)\subseteq \hat{M}_t \}
\end{align*}
which we give the subspace topology inherited from the product topology on $C_k(\hat{M})\times [0,\infty)$ and $\Gamma(\hat{M})\times [0,\infty)$ respectively.

To justify this definition we will construct homotopy equivalences
\[
\ibar_C : C_k(M) \longleftrightarrow C_k(M,e) : \jbar_C \qquad\text{and}\qquad \ibar_\Gamma : \Gamma_k(M) \longleftrightarrow \Gamma_k(M,e) : \jbar_\Gamma
\]
for any such $(M,e)$. First choose a continuous map $\phi\colon [0,\infty) \to \Emb((-3,\infty),(-3,\infty))$ such that $\phi_t((-3,\infty)) = (-3,t)$ and $\phi_t$ is the identity near $-3$. This induces a continuous map $\bar{\phi}\colon [0,\infty) \to \Emb(\hat{M},\hat{M})$ via
\[
\bar{\phi}_t(p) = \begin{cases}
p & p\in \hat{M} \smallsetminus U \\
(q,\phi_t(s)) & p=(q,s) \in U,
\end{cases}
\]
where $U = (-1,1)^{m-1}\times (-3,\infty)$. Also choose a trivialization $\psi\colon \dot{T}U \to U\times S^d$ over $U$. Then define $\ibar_C(c) = (\bar{\phi}_0 (e(c)),0)$ and $\jbar_C(c,t) = e^{-1}(\bar{\phi}_t^{-1}(c))$. The composition $\jbar_C \circ \ibar_C$ is the identity and a homotopy $\id \Rightarrow \ibar_C \circ \jbar_C$ is given by $(s,c,t)\mapsto (\bar{\phi}_{st}(\bar{\phi}_t^{-1}(c)),st)$. So $\ibar_C$ and $\jbar_C$ are homotopy inverse, as required. We may similarly define $\ibar_\Gamma(\sigma) = (\sigma^\prime,0)$ where
\[
\sigma^\prime = \begin{cases}
\sigma & \text{on } \hat{M} \smallsetminus U \\
\psi^{-1} \circ (\bar{\phi}_0 \times \mathrm{id}) \circ \psi \circ \sigma \circ \bar{\phi}_0^{-1} & \text{on } (-1,1)^{m-1} \times (-3,0) \\
\infty & \text{on } (-1,1)^{m-1} \times [0,\infty)
\end{cases}
\]
and
\[
\jbar_\Gamma(\sigma,t) = \begin{cases}
\sigma & \text{on } \hat{M} \smallsetminus U \\
\psi^{-1} \circ (\bar{\phi}_t^{-1} \times \mathrm{id}) \circ \psi \circ \sigma \circ \bar{\phi}_t & \text{on } U
\end{cases}
\]
where we are implicitly identifying $M$ with $\hat{M}=e(M)$. The composition $\jbar_\Gamma \circ \ibar_\Gamma$ is the identity and a homotopy $\id \Rightarrow \ibar_\Gamma \circ \jbar_\Gamma$ is given by
\[
(s,\sigma,t) \mapsto \begin{cases}
\sigma & \text{on } \hat{M} \smallsetminus U \\
\psi^{-1} \circ (\bar{\phi}_{st}\times \mathrm{id}) \circ (\bar{\phi}_t^{-1} \times \mathrm{id}) \circ \psi \circ \sigma \circ \bar{\phi}_t \circ \bar{\phi}_{st}^{-1} & \text{on } (-1,1)^{m-1} \times (-3,st) \\
\infty & \text{on } (-1,1)^{m-1} \times [st,\infty).
\end{cases}
\]
Hence $\ibar_\Gamma$ and $\jbar_\Gamma$ are also homotopy inverse, as required.

\subsection{Stabilization and scanning maps.}\label{sssStabScanMaps}

We now define the stabilization maps for configuration and section spaces, and two models for the scanning map, each depending on some auxiliary data. The stabilization map for configuration spaces depends on an embedding $e$ of $M$ into $(-1,1)^d \times \bR$ which is prescribed at one end (so in particular $M$ must be open). The first model for the scanning map depends on a Riemannian metric $g$ and a function $\rho\colon M\to (0,1)$, and is defined for both open and closed manifolds. Finally, the stabilization map for section spaces and the second model for the scanning map depend on choices of $e$, $g$ and $\rho$ (so $M$ must be open).

We first define the stabilization map for configuration spaces, assuming that we have an embedding $e\colon M\emb (-1,1)^d \times \bR$ satisfying \eqref{eEmbeddingCondition}. As before we write $\hat{M} = e(M)$.

\begin{definition}\label{dStabilizationMapConfig}
Define the stabilization map $t\colon C_k(M,e) \to C_{k+1}(M,e)$ by
\begin{equation}\label{eStabMapConcrete}
(c,t) \mapsto (c\cup\{(\mathbf{0},t+\tfrac12)\},t+1),
\end{equation}
where $\mathbf{0}$ denotes $(0,\dotsc,0)\in (-1,1)^d$. Using the identifications of \S\ref{sssConfigSectionSpaces} this also determines a ``stabilization map'' $\jbar_C \circ t\circ \ibar_C \colon C_k(M)\to C_{k+1}(M)$ for any smooth, connected, non-compact manifold $M$, depending on the choice of embedding $e$.
\end{definition}

We next define the scanning map for any Riemannian manifold with a choice of lower bound for its injectivity radius. Let $g$ be a Riemannian metric on $M$ and let $\rho\colon M\to (0,1)$ be a strict lower bound for its injectivity radius. Let $\rmd$ denote the metric (in the sense of a distance function) induced by $g$ and for $p\in M$ and $r>0$ define $B_r(p)$ to be $\{ q\in M \mid \rmd(p,q)<r \}$. If $r\in (0,\rho(p)]$ the exponential map for $M$ restricts to a homeomorphism from the open $r$-ball in $T_p M$ to $B_r(p)$. Composing the inverse of this with the dilation $v\mapsto (\frac{1}{r-\lvert v\rvert} - \frac{1}{r}).v$ defines a homeomorphism
\[
\cE_{p,r}\colon B_r(p) \to T_p M.
\]
We also define a function $\epsilon\colon C_k(M)\times M \to [0,1]$ by:
\[
\epsilon(c,p) \coloneqq \mathrm{sup}\{ r\in [0,1] \mid r\leq\rho(p) \text{ and } \lvert B_r(p)\cap c\rvert \leq 1 \}.
\]
Note that the subset $\{ r\in [0,1] \mid r\leq\rho(p) \text{ and } \lvert B_r(p)\cap c\rvert \leq 1 \}$ of $[0,1]$ is closed and hence compact, so $\epsilon$ is a continuous function and $\lvert B_{\epsilon(c,p)}(p)\cap c\rvert \leq 1$ for all $(c,p)$. Also note that $\epsilon(c,p)>0$ for all $(c,p)$.

\begin{definition}\label{dScanningMap1}
Define the scanning map $s\colon C_k(M) \to \Gamma_k(M)$ by $s(c)=\sigma$ where
\begin{equation}\label{eScanningDef}
\sigma(p) = \begin{cases}
\cE_{p,\epsilon(c,p)}(q) \in T_p M & \text{if } B_{\epsilon(c,p)}(p)\cap c = \{q\} \\
\infty & \text{if } B_{\epsilon(c,p)}(p)\cap c = \varnothing.
\end{cases}
\end{equation}
\end{definition}

We now define a second model for the scanning map, of the form $s\colon C_k(M,e) \to \Gamma_k(M,e)$, depending on some auxiliary data $(e,g,\rho)$. Choose an embedding $e\colon M\emb (-1,1)^{m-1}\times \bR$ satisfying \eqref{eEmbeddingCondition}, a Riemannian metric $g$ on $M$ and a smooth function $\rho\colon M\to (0,1)$ which is a strict lower bound for the injectivity radius of $(M,g)$. The extra property which we require this to satisfy is that the radius-$\rho$ balls in $M$ do not extend too far in the $(d+1)$st coordinate direction, which is formulated precisely as follows. As previously, we write $U = (-1,1)^{m-1}\times (-3,\infty)$ for the chosen end of the manifold $\hat{M}=e(M)$ and implicitly identify $M$ with $\hat{M}$. For points $p\in U$ we require that
\begin{equation}\label{eCondition1}
B_{\rho(p)}(p) \subseteq (-1,1)^d \times (x-\tfrac18,x+\tfrac18)
\end{equation}
where $x=\pi_{d+1}(p)$ and for points $p\in \hat{M}\smallsetminus U$ we require that
\begin{equation}\label{eCondition2}
B_{\rho(p)}(p) \subseteq (-1,1)^d \times (-\infty,-3+\tfrac18).
\end{equation}
Definition \ref{dScanningMap1} applied to $\hat{M}$ and the data $(g,\rho)$ (transferred to $\hat{M}$ via $e$) gives a scanning map which we denote by $\hat{s}\colon C_k(\hat{M}) \to \Gamma_k(\hat{M})$.

\begin{definition}\label{dScanningMap2}
Define the model $s\colon C_k(M,e) \to \Gamma_k(M,e)$ for the scanning map by
\[
(c,t) \mapsto (\hat{s}(c),t+\tfrac14).
\]
\end{definition}

To describe a homotopy filling the square \eqref{eTwoScanningMaps} we first define a ``shifting'' operation on compactly-supported sections of $\hat{M}$ which roughly conjugates a section by the self-embedding $\bar{\phi}_t\colon \hat{M}\emb \hat{M}$. For $t\in [0,\infty)$ and $\sigma\in \Gamma(\hat{M})$ define
\[
\Phi_t(\sigma) = \begin{cases}
\sigma & \text{on } \hat{M} \smallsetminus U \\
\psi^{-1} \circ (\bar{\phi}_t \times \mathrm{id}) \circ \psi \circ \sigma \circ \bar{\phi}_t^{-1} & \text{on } (-1,1)^{m-1} \times (-3,t) \\
\infty & \text{on } (-1,1)^{m-1} \times [t,\infty).
\end{cases}
\]
For example in \S\ref{sssConfigSectionSpaces} we have $\ibar_C(\sigma) = (\Phi_0(\sigma),0)$. Now, the two ways around the square \eqref{eTwoScanningMaps} take a configuration $c\in C_k(M)$ to
\[
(\hat{s}(\bar{\phi}_0(c)),\tfrac14) \qquad\text{and}\qquad (\Phi_0(\hat{s}(c)),0)
\]
respectively. To fill in a homotopy between these maps one can define it to take a configuration $c\in C_k(M)$ to
\[
(\Phi_{\mathrm{tanh}(t)}(\hat{s}(\bar{\phi}_{\mathrm{tanh}(1-t)}(c))),\tfrac{t}{4})
\]
for $t\in (0,1)$. More informally, the two routes around \eqref{eTwoScanningMaps} either scan the configuration and then compress the resulting section so that its support lies in $\hat{M}_0$, or conversely compress the configuration so that it lies in $\hat{M}_0$ and then scan it. The homotopy interpolates between these two by compressing the configuration into $\hat{M}_t$ for $t\in (0,\infty)$, scanning and then further compressing the resulting section into $\hat{M}_0$.

\begin{definition}\label{dStabilizationMapSection}
Suppose we have an embedding $e\colon M\emb (-1,1)^d\times \bR$ satisfying \eqref{eEmbeddingCondition}, as well as a Riemannian metric $g$ on $M$ and a strict lower bound $\rho\colon M\to (0,1)$ for its injectivity radius satisfying \eqref{eCondition1} and \eqref{eCondition2}. Define the stabilization map $T\colon \Gamma_k(M,e)\to \Gamma_{k+1}(M,e)$ by $T(\sigma,t)=(\sigma^\prime,t+1)$ where
\[
\sigma^\prime(p) = \begin{cases}
\sigma(p) & \text{if } p\in \hat{M}_t \\
\cE_{p,\rho(p)}(q) & \text{if } q\in B_{\rho(p)}(p) \\
\infty & \text{otherwise},
\end{cases}
\]
where $q = (\mathbf{0},t+\frac14)$. Using the identifications of \S\ref{sssConfigSectionSpaces} this also determines a ``stabilization map'' $\jbar_\Gamma \circ T\circ \ibar_\Gamma \colon \Gamma_k(M) \to \Gamma_{k+1}(M)$ for any smooth, connected, non-compact manifold $M$, depending on the choice of embedding $e$.
\end{definition}

\paragraph{Reverse stabilization.}
In the definition above, the restriction of $\sigma^\prime$ to the strip
\[
\hat{M}_{t+1} \smallsetminus \hat{M}_t = (-1,1)^{m-1}\times [t,t+1)
\]
has degree $1$. We could insert a reflection into the definition of $\sigma^\prime$ so that it instead has degree $-1$, and thereby define a ``reverse stabilization map'' $U\colon \Gamma_k(M,e) \to \Gamma_{k-1}(M,e)$. It is not hard to show that $U$ is a homotopy inverse for $T$, so in particular the homotopy type of $\Gamma_k(M,e)$ is independent of $k$.

\subsection{Oriented configuration spaces.}\label{sssOriented}

Recall that the oriented configuration spaces are defined by $C_k^+(M) = (M^k \smallsetminus \Delta_f)/A_k$, and double-cover the unordered configuration spaces $C_k(M)$. The homotopy equivalence $C_k(M) \simeq C_k(M,e)$ constructed in \S\ref{sssConfigSectionSpaces} clearly lifts to a homotopy equivalence $C_k^+(M) \simeq C_k^+(M,e)$, where $C_k^+(M,e)$ is the double cover of $C_k(M,e)$ defined by
\[
C_k^+(M,e) \coloneqq \{ (c,t) \in C_k^+(\hat{M})\times [0,\infty) \mid c\subseteq \hat{M}_t \}.
\]
There are exactly two lifts of the stabilization map $t\colon C_k(M,e) \to C_{k+1}(M,e)$ to a map $C_k^+(M,e) \to C_{k+1}^+(M,e)$; they differ by the deck transformation which reverses the orientation of a configuration, in other words which permutes an ordered configuration representing an element of $C_{k+1}^+(M,e)$ by an odd permutation. We choose $t^+\colon C_k^+(M,e) \to C_{k+1}^+(M,e)$ to be the one defined by
\[
([p_1,\dotsc,p_k],t) \mapsto ([p_1,\dotsc,p_n,(\mathbf{0},t+\tfrac12)],t+1)
\]
where $[-]$ denotes a finite ordered set modulo even permutations.

\subsection{Scanning commutes with stabilization.}\label{sssCommutativity}

The stabilization and scanning maps were defined in \S\ref{sssStabScanMaps} in order to commute on the nose, as the following lemma proves.

\begin{lemma}
The following square commutes:
\begin{equation}\label{eStabScan}
\centering
\begin{split}
\begin{tikzpicture}
[x=1mm,y=1mm]
\node (tl) at (0,15) {$C_k(M,e)$};
\node (tr) at (30,15) {$C_{k+1}(M,e)$};
\node (bl) at (0,0) {$\Gamma_k(M,e)$};
\node (br) at (30,0) {$\Gamma_{k+1}(M,e)$};
\draw[->] (tl) to node[above,font=\small]{$t$} (tr);
\draw[->] (bl) to node[below,font=\small]{$T$} (br);
\draw[->] (tl) to node[left,font=\small]{$s$} (bl);
\draw[->] (tr) to node[right,font=\small]{$s$} (br);
\end{tikzpicture}
\end{split}
\end{equation}
\end{lemma}

An illustration of the two ways around this square is given in Figure \ref{fStabScan}.

\begin{proof}
Let $(c,t)\in C_k(M,e)$. Then
\begin{align*}
t(c,t) &= (c_1,t+1) & s(c_1,t+1) &= (\sigma_1,t+\tfrac54) \\
s(c,t) &= (\sigma_0,t+\tfrac14) & T(\sigma_0,t+\tfrac14) &= (\sigma_2,t+\tfrac54)
\end{align*}
where $c_1 = c\cup\{q\}$ for $q=(\mathbf{0},t+\frac12)$. We need to show that $\sigma_1=\sigma_2$.

First consider $p\in \hat{M}\smallsetminus \hat{M}_{t+\frac14}$. In this case we have $B_{\rho(p)}(p) \subseteq \hat{M}\smallsetminus \hat{M}_t$ by \eqref{eCondition1}, so $B_{\rho(p)}(p) \cap c_1$ can only be either $\{q\}$ or $\varnothing$. Hence $\epsilon(c_1,p)=\rho(p)$ and
\[
\sigma_1(p) = \hat{s}(c_1)(p) = \begin{cases}
\cE_{p,\rho(p)}(q) & \text{if } q\in B_{\rho(p)}(p) \\
\infty & \text{otherwise}.
\end{cases}
\]
This is also the value of $\sigma_2(p)$, by definition of $T$, so $\sigma_1$ and $\sigma_2$ agree on this subspace.

For $p\in \hat{M}_{t+\frac14}$ the value of $\sigma_2(p) = \sigma_0(p)$ is given by \eqref{eScanningDef} with $M\mapsto\hat{M}$, and $\sigma_1(p)$ is given by the same formula with $c$ replaced by $c_1$. To show that these agree it suffices to show that $\epsilon(c,p) = \epsilon(c_1,p)$ and $B_{\epsilon(c,p)}(p)\cap c = B_{\epsilon(c,p)}(p)\cap c_1$. Note that the value of $\epsilon(c,p)$ depends only on $p$ and the subconfiguration $B_{\rho(p)}(p)\cap c$ of $c$. The symmetric difference of $c$ and $c_1$ is $\{q\}$, but by \eqref{eCondition1} we have $q\notin B_{\rho(p)}(p)$. Hence we have
\[
B_{\rho(p)}(p)\cap c = B_{\rho(p)}(p)\cap c_1,
\]
so $\epsilon(c,p)=\epsilon(c_1,p)$. Since $\epsilon(c,p)\leq \rho(p)$ we also have $q\notin B_{\epsilon(c,p)}(p)$ and so
\[
B_{\epsilon(c,p)}(p)\cap c = B_{\epsilon(c,p)}(p)\cap c_1.\qedhere
\]
\end{proof}

\begin{figure}[ht]
\centering
\begin{tikzpicture}
[x=1mm,y=1mm,font=\footnotesize,scale=1.5]
\fill[black!20] (-10,0) rectangle (0,10);
\draw (-10,0)--(20,0);
\draw (-10,10)--(20,10);
\foreach \x in {0,10} \node at (\x,0) [fill,inner sep=1pt]{};
\node at (-10,-2) [anchor=east] {$t+\ldots$};
\node at (0,-2) {$0$};
\node at (10,-2) {$1$};
\draw (0,0)--(0,10);
\begin{scope}[xshift=45mm]
\fill[black!20] (-10,0) rectangle (0,10);
\draw (-10,0)--(20,0);
\draw (-10,10)--(20,10);
\foreach \x in {0,5,10} \node at (\x,0) [fill,inner sep=1pt]{};
\node at (0,-2) {$0$};
\node at (5,-2) {$\frac12$};
\node at (10,-2) {$1$};
\draw (10,0)--(10,10);
\node at (5,5) [fill,inner sep=1pt]{};
\end{scope}
\begin{scope}[yshift=-25mm]
\fill[black!20] (-10,0) rectangle (1.25,10);
\draw (-10,0)--(20,0);
\draw (-10,10)--(20,10);
\foreach \x in {0,1.25,2.5,10} \node at (\x,0) [fill,inner sep=1pt]{};
\node at (-10,-2) [anchor=east] {$t+\ldots$};
\node at (0,-2) {$0$};
\node at (1.25,-2) {$\frac18$};
\node at (2.5,-2) {$\frac14$};
\node at (10,-2) {$1$};
\draw (2.5,0)--(2.5,10);
\end{scope}
\begin{scope}[xshift=45mm,yshift=-25mm]
\fill[black!20] (-10,0) rectangle (1.25,10);
\fill[black!20] (5,5) circle (1.25);
\draw (-10,0)--(20,0);
\draw (-10,10)--(20,10);
\foreach \x in {0,1.25,5,10,12.5} \node at (\x,0) [fill,inner sep=1pt]{};
\node at (0,-2) {$0$};
\node at (1.25,-2) {$\frac18$};
\node at (5,-2) {$\frac12$};
\node at (10,-2) {$1$};
\node at (12.5,-2) {$\frac54$};
\draw (12.5,0)--(12.5,10);
\end{scope}
\node at (27.5,5) {$\rightsquigarrow$};
\node at (27.5,-20) {$\rightsquigarrow$};
\node at (5,-7.5) {\rotatebox{270}{$\rightsquigarrow$}};
\node at (50,-7.5) {\rotatebox{270}{$\rightsquigarrow$}};
\end{tikzpicture}
\caption{The effect of the maps of the square \eqref{eStabScan} on a point in $C_k(M,e)$. The pictures show the segment $(-1,1)^d\times (t-1,t+2)$ of $\hat{M}$. The configuration (respectively support of the section) is contained in the shaded region, except in the top-right where there is one extra configuration point as indicated.}\label{fStabScan}
\end{figure}
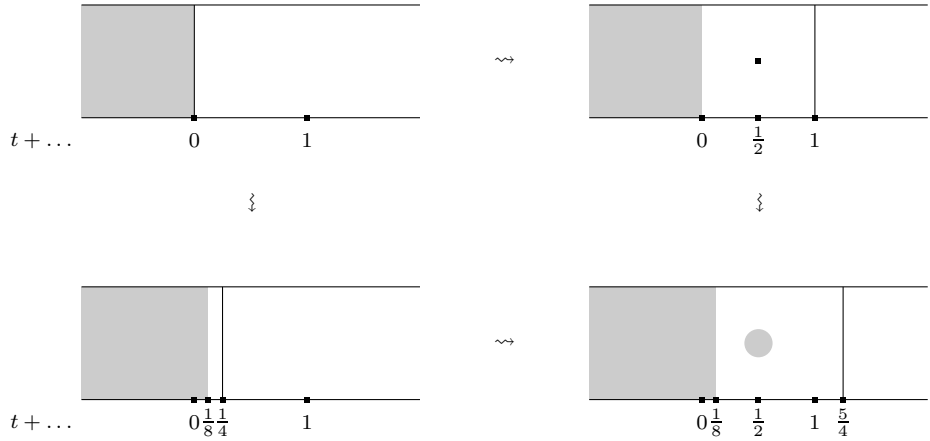

\subsection{A double cover of the section space.}\label{sssDoubleCover}

We now geometrically construct a homomorphism $\pi_1(\Gamma_k(M,e)) \to \bZ/2$. This defines double covers of all the spaces in the square \eqref{eTwoScanningMaps}, and it will be constructed so that the double cover of $C_k(M)$ that it determines is the oriented configuration space $C_k^+(M)$.

To do this we construct a map $e_\Gamma \colon \Gamma_k(M,e) \to \Gamma_k((-1,1)^d \times \bR)$ such that the square
\begin{equation}\label{eDoubleCoverSquare}
\centering
\begin{split}
\begin{tikzpicture}
[x=1mm,y=1mm]
\node (tl) at (0,15) {$C_k(M,e)$};
\node (tr) at (40,15) {$C_k((-1,1)^d \times \bR)$};
\node (bl) at (0,0) {$\Gamma_k(M,e)$};
\node (br) at (40,0) {$\Gamma_k((-1,1)^d \times \bR)$};
\draw[->] (tl) to node[above,font=\small]{$e_C$} (tr);
\draw[->] (bl) to node[below,font=\small]{$e_\Gamma$} (br);
\draw[->] (tl) to node[left,font=\small]{$s$} (bl);
\draw[->] (tr) to node[right,font=\small]{$s$} (br);
\end{tikzpicture}
\end{split}
\end{equation}
commutes up to homotopy, where the left-hand $s$ is as in Definition \ref{dScanningMap2}, the right-hand $s$ is as in Definition \ref{dScanningMap1} and $e_C(c,t)=c$. This will achieve the aim of this subsection by the following. The bottom right-hand space is homotopy equivalent to $\Omega_k^{d+1}S^{d+1}$ so since $d\geq 2$ it has fundamental group $\bZ/2$. The map on $\pi_1$ induced by $e_\Gamma$ is then the desired homomorphism $\pi_1(\Gamma_k(M,e)) \to \bZ/2$, defining a double cover $\Gamma_k^+(M,e)$ of $\Gamma_k(M,e)$. We want the pullback of this double cover along $s\circ \ibar_C\colon C_k(M) \to C_k(M,e) \to \Gamma_k(M,e)$ to be the oriented configuration space $C_k^+(M)$. Note that on fundamental groups the right-hand vertical map is the sign homomorphism $\Sigma_k \to \bZ/2$,\footnote{Assuming that $d\geq 2$ and $k\geq 2$. The former ensures that $\pi_1(C_k(\bR^{d+1})) \cong \Sigma_k$ and $\pi_1(\Omega_{\bullet}^{d+1} S^{d+1}) \cong \bZ/2$. The latter plus McDuff's theorem (Theorem \ref{scanrange}) implies that the map is surjective on $H_1$, and therefore also on $\pi_1$ since $\pi_1$ of the codomain is abelian. The only surjective map $\Sigma_k \to \bZ/2$ is the sign homomorphism.} so this is isomorphic to the pullback along $e_C \circ \ibar_C$ of the double cover of $C_k((-1,1)^d \times \bR)$ corresponding to the alternating group $A_k \leq \Sigma_k$. Hence it is indeed the oriented configuration space.

It thus remains to construct $e_\Gamma$ and see that \eqref{eDoubleCoverSquare} commutes up to homotopy. We will construct $e_\Gamma$ in detail but leave the homotopy-commutativity of the square as an exercise for the reader, since the proofs of our main results do not depend on the equivalence of the two definitions of $\Gamma_k^+(M)$.

\begin{definition}
For brevity write $E = (-1,1)^d \times \bR$, so we have an embedding $e\colon M\emb E$, and let $\pi \colon \nu \to M$ denote the normal bundle of this embedding. Choose a tubular neighborhood for $e$, in other words an embedding $\bar{e}\colon \nu \hookrightarrow E$ so that $\bar{e}\circ z = e$, where $z$ is the zero section of the normal bundle, also satisfying the following additional condition. The restrictions of tangent bundles $T\nu|_{z(M)}$ and $TE|_{e(M)}$ are both canonically isomorphic to $\nu \oplus TM$ and the bundle map $T\bar{e}\colon T\nu|_{z(M)} \to TE|_{e(M)}$ is of the form $\phi\oplus\mathrm{id}$ under these identifications for some bundle automorphism $\phi$ of $\nu$; we require that $\phi=\mathrm{id}$.

Let $(\sigma,t)\in \Gamma_k(M,e)$, so in particular $\sigma\in\Gamma(\hat{M})$. We then define
\[
e_\Gamma(\sigma,t) \coloneqq \sigma^\prime \in \Gamma(E)
\]
as follows. For $p\in E$ not in the tubular neighborhood $\bar{e}(\nu)$ let $\sigma^\prime(p)=\infty$. Now let $p=\bar{e}(x)$ for some $x\in \nu$. We want to define $\sigma^\prime(p)$ in $\dot T_p E$, which is identified via $\bar{e}$ with $\dot T_x \nu$. Note that the bundle $T\nu$ decomposes as $\pi^* \nu \oplus \pi^* TM$ so what we need is an element of the one-point compactification of $\nu_{\pi(x)} \oplus T_{\pi(x)}M$. But we have the element
\begin{equation}\label{ePartialCompactification}
(-x,\sigma(\pi(x))) \in \nu_{\pi(x)} \oplus \dot T_{\pi(x)} M
\end{equation}
which lives in a one-point partial compactification of $\nu_{\pi(x)} \oplus T_{\pi(x)}M$, and therefore also gives an element in its one-point compactification. This defines a section $\sigma^\prime \in \Gamma(E)$ and one can check that $\deg(\sigma^\prime) = \deg(\sigma)$ so in fact $\sigma^\prime \in \Gamma_k(E)$.
\end{definition}

This is the formal definition of $\sigma^\prime$. The idea is to extend the section $\sigma$ to the tubular neighborhood $\bar{e}(\nu)$ so that its value goes to infinity as you go to infinity in the fiber direction of the normal bundle $\nu$, so that we can then extend it by infinity outside the tubular neighborhood. To achieve this, we define $\sigma^\prime(x)$, for a point $x$ in the normal bundle over $\pi(x)\in M$, to be the value of $\sigma$ on $\pi(x)$ plus another term which goes to infinity as $x$ goes to infinity. Using the decomposition of the tangent bundle of the normal bundle explained above, we can take this extra term to be $-x$. See Figure \ref{fExtendingSection}.

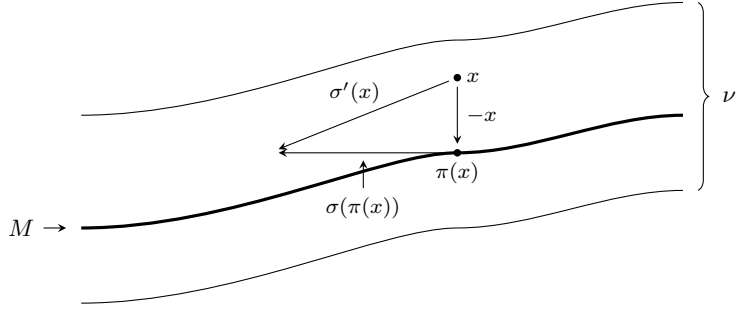
\begin{figure}[ht]
\centering
\begin{tikzpicture}
[x=1mm,y=1mm,font=\small,>=stealth]
\foreach \y in {15,-10}
  {
  \draw (0,\y) .. controls (10,\y) and (20,\y+5) .. (30,\y+5);
  \draw (-50,\y-10) .. controls (-30,\y-10) and (-10,\y) .. (0,\y);
  }
\draw[very thick] (0,0) .. controls (10,0) and (20,5) .. (30,5);
\draw[very thick] (-50,-10) .. controls (-30,-10) and (-10,0) .. (0,0);
\draw[decorate,decoration={brace,amplitude=4pt}] (32,20)--(32,-5);
\node at (34,7.5) [anchor=west] {$\nu$};
\node (x) at (-55,-10) [anchor=east] {$M$};
\draw[->] (x.east) to (-52,-10);
\node (a) at (0,10) [fill,circle,inner sep=1pt,outer sep=2pt] {};
\node (b) at (0,0) [fill,circle,inner sep=1pt,outer sep=2pt] {};
\node (c) at (-25,0) {};
\draw[->] (a) to (b);
\draw[->] (b) to (c);
\draw[->] (a) to (c);
\node at (0,0) [anchor=north,font=\footnotesize] {$\pi(x)$};
\node at (0,10) [anchor=west,font=\footnotesize] {$x$};
\node at (0,5) [anchor=west,font=\footnotesize] {$-x$};
\node (y) at (-12.5,-10) [anchor=south,font=\footnotesize] {$\sigma(\pi(x))$};
\draw[->] (y.north) to (-12.5,-1);
\node at (-9,8) [anchor=east,font=\footnotesize] {$\sigma^\prime(x)$};
\end{tikzpicture}
\caption{The definition of $\sigma^\prime$ on a point $x$ in the normal bundle $\nu$. This defines $\sigma^\prime$ on the tubular neighborhood $\bar{e}(\nu)$, and outside this tubular neighborhood $\sigma^\prime$ is defined to be $\infty$.}\label{fExtendingSection}
\end{figure}

\section{Euclidean spaces}\label{ss-r2crossn}

Let $(M,e)$ be as in \S\ref{ss-concrete-models}. The sequences of stabilization maps for configuration and section spaces for $k\geq 0$ determine mapping telescopes $\tel(C_k(M,e))$ and $\tel(\Gamma_k(M,e))$, and since the scanning maps commute with stabilization maps there is an induced map of mapping telescopes $s\colon \tel(C_k(M,e)) \to \tel(\Gamma_k(M,e))$. The precise statement of Theorem \ref{acyclic} is that this map is acyclic.

Write $C(M,e) = \bigsqcup_{k\in\bN}C_k(M,e)$ and $\Gamma(M,e) = \bigsqcup_{k\in\bZ}\Gamma_k(M,e)$. The stabilization maps $t$ and $T$ are endomorphisms of $C(M,e)$ and $\Gamma(M,e)$ respectively, and we can also form mapping telescopes $\tel(C(M,e))$ and $\tel(\Gamma(M,e))$ with respect to these endomorphisms. The scanning map $s\colon C(M,e) \to \Gamma(M,e)$ commutes with these endomorphisms and therefore induces a map of mapping telescopes $s_\tel \colon \tel(C(M,e)) \to \tel(\Gamma(M,e))$. There is a commutative square
\begin{equation}\label{eTelescopeSquare}
\centering
\begin{split}
\begin{tikzpicture}
[x=1mm,y=1mm]
\node (tl) at (0,15) {$\tel(C(M,e))$};
\node (tr) at (40,15) {$\bZ\times \tel(C_k(M,e))$};
\node (bl) at (0,0) {$\tel(\Gamma(M,e))$};
\node (br) at (40,0) {$\bZ\times \tel(\Gamma_k(M,e))$};
\incl{(tl)}{(tr)}
\draw[->] (bl) to (br);
\draw[->] (tl) to node[left,font=\small]{$s_\tel$} (bl);
\draw[->] (tr) to node[right,font=\small]{$1\times s$} (br);
\end{tikzpicture}
\end{split}
\end{equation}
in which the horizontal maps are homotopy equivalences, so it is equivalent to prove that the map $s_\tel \colon \tel(C(M,e)) \to \tel(\Gamma(M,e))$ is acyclic.

In this section we prove this in the special case $(M,e) = (E^m,\iota)$ for $d\geq m\geq 2$, where $E^m=(-1,1)^{m-1}\times\bR$ and $\iota$ is the inclusion of $E^m$ into $(-1,1)^d\times\bR$. The main technical input for this is the following.

\begin{lemma}\label{lSimple}
The spaces $\Gamma(E^m,\iota)$ and $\tel(C(E^m,\iota))^+$ are simple, where $(-)^+$ denotes Quillen plus-construction with respect to the maximal perfect subgroup.
\end{lemma}

Before proving this we show how it implies Theorem \ref{acyclic} in the special case.

\begin{proof}[Proof of Theorem \ref{acyclic} when $(M,e)=(E^m,\iota)$]
Recall from \S\ref{sssStabScanMaps} that the stabilization maps $T\colon \Gamma_k(M,e) \to \Gamma_{k+1}(M,e)$ are homotopy equivalences for any $(M,e)$, so we have a homotopy equivalence $\tel(\Gamma(M,e)) \simeq \Gamma(M,e)$. Also, by McDuff's theorem \cite[Theorem 1.2]{Mc1} (Theorem \ref{scanlim} in the introduction) the map
\begin{equation}\label{eStabilizationTel}
s_\tel \colon \tel(C(M,e)) \longrightarrow \tel(\Gamma(M,e))
\end{equation}
is a homology equivalence (for untwisted integral coefficients) for any $(M,e)$.

Lemma \ref{lSimple} tells us that $\tel(\Gamma(E^m,\iota))$ and $\tel(C(E^m,\iota))^+$ are simple spaces, so applying Quillen's plus-construction to the map \eqref{eStabilizationTel} for $(M,e)=(E^m,\iota)$ results in a homology equivalence between simple spaces. This is a weak equivalence by the homology Whitehead theorem, so the original map \eqref{eStabilizationTel} is acyclic.
\end{proof}

\begin{proof}[Proof of Lemma \ref{lSimple}]
We will write $J=(-1,1)$. Recall that
\begin{align*}
C_k(E^m,\iota) &= \{ (c,t)\in C_k(J^{m-1}\times\bR)\times [0,\infty) \mid c\subseteq J^{m-1}\times (-\infty,t) \} \\
\Gamma_k(E^m,\iota) &= \{ (\sigma,t)\in \Gamma_k(J^{m-1}\times\bR)\times [0,\infty) \mid \mathrm{supp}(\sigma)\subseteq J^{m-1}\times (-\infty,t) \}
\end{align*}
We define subspaces
\begin{align*}
\cM_k &\coloneqq \{ (c,t)\in C_k(J^{m-1}\times\bR)\times [0,\infty) \mid c\subseteq J^{m-1}\times (0,t) \} \\
\cN_k &\coloneqq \{ (\sigma,t)\in \Gamma_k(J^{m-1}\times\bR)\times [0,\infty) \mid \mathrm{supp}(\sigma)\subseteq J^{m-1}\times (0,t) \}
\end{align*}
and let $\cM = \bigsqcup_{k\in\bN}\cM_k$ and $\cN = \bigsqcup_{k\in\bZ}\cN_k$.

For $t\in\bR$ define $\lambda_t$ to be the automorphism $(x,u)\mapsto (x,u+t)$ of $J^{m-1}\times\bR$. Also choose a trivialization of the bundle $\dot{T}(J^{m-1}\times\bR)$, i.e.\ a homeomorphism $\psi\colon \dot{T}(J^{m-1}\times\bR) \to J^{m-1}\times\bR\times S^m$ over $J^{m-1}\times\bR$. Then we can define monoid structures on $\cM$ and $\cN$ by
\begin{align*}
(c_1,t_1)\cdot (c_2,t_2) &= (c_1 \cup \lambda_{t_1}(c_2),t_1+t_2) \\
(\sigma_1,t_1)\cdot (\sigma_2,t_2) &= (\sigma_3,t_1+t_2)
\end{align*}
where $\sigma_3$ is defined by
\[
\sigma_3 = \begin{cases}
\sigma_1 & \text{on } J^{m-1}\times (0,t_1) \\
\psi^{-1}\circ (\lambda_{t_1}\times\mathrm{id}) \circ \psi \circ \sigma_2 \circ \lambda_{-t_1} & \text{on } J^{m-1}\times (t_1,t_1+t_2) \\
\infty & \text{otherwise}.
\end{cases}
\]
Note that $\pi_0(\cM) = \bN$ and $\pi_0(\cN) = \bZ$, so $\cN$ is grouplike.

The inclusions $\cM_k \emb C_k(E^m,\iota)$ and $\cN_k \emb \Gamma_k(E^m,\iota)$ are homotopy equivalences; for example one can define homotopy inverses as follows. Choose an order-preserving homeomorphism $\phi\colon (-1,0) \to (-\infty,0)$ and for $t\in (0,\infty)$ define a homeomorphism $\bar{\phi}_t\colon J^{m-1}\times (0,t) \to J^{m-1}\times (-\infty,t)$ by $\bar{\phi}_t(x,u) = (x,\phi(\frac{u}{t}-1)+t)$. We can then define homotopy inverses by $(c,t) \mapsto (\bar{\phi}_t^{-1}(c),t)$ and $(\sigma,t) \mapsto (\sigma^\prime,t)$, where
\[
\sigma^\prime = \begin{cases}
\psi^{-1}\circ (\bar{\phi}_t^{-1}\times\mathrm{id})\circ \psi \circ \sigma \circ \bar{\phi}_t & \text{on } J^{m-1}\times (0,t) \\
\infty & \text{otherwise}.
\end{cases}
\]

Hence we have a homotopy equivalence $\cN \simeq \Gamma(E^m,\iota)$. Since $\cN$ is a grouplike monoid, in particular a grouplike H-space, it is simple, and therefore so is $\Gamma(E^m,\iota)$.

The monoid $\cM$ is not grouplike, but it is homotopy-commutative since $m\geq 2$ (this can be shown analogously to the proof that higher homotopy groups are abelian). Let $m_1\in\cM_1$ be the element $(\{(\mathbf{0},\frac12)\},1)$. Note that the restriction of the stabilization map $t\colon C_k(E^m,\iota) \to C_{k+1}(E^m,\iota)$ to $\cM_k$ is precisely right-multiplication by $m_1$. Hence the square
\begin{center}
\begin{tikzpicture}
[x=1mm,y=1mm]
\node (tl) at (0,12) {$C(E^m,\iota)$};
\node (tr) at (30,12) {$C(E^m,\iota)$};
\node (bl) at (0,0) {$\cM$};
\node (br) at (30,0) {$\cM$};
\incl{(bl)}{(tl)}
\incl{(br)}{(tr)}
\draw[->] (tl) to node[above,font=\small]{$t$} (tr);
\draw[->] (bl) to node[below,font=\small]{$-\cdot m_1$} (br);
\end{tikzpicture}
\end{center}
commutes (and the vertical maps are homotopy equivalences), so we have a homotopy equivalence $\tel(\cM) \emb \tel(C(E^m,\iota))$. By Proposition \ref{cGroupCompletion2}, the Quillen plus-construction $\tel(\cM)^+$ is a simple space. Hence $\tel(C(E^m,\iota))^+$ is also simple.
\end{proof}

%
%

\section{Manifolds admitting boundary}\label{ss-open}

We begin with some notation for relative configuration and section spaces. Let $N$ be a smooth manifold with boundary $\partial N$ and let $A\subseteq N$ be a subset.

\begin{definition}\label{dRelativeConfigSection}
Define the relative configuration space $C(N,A)$ to be the quotient of $C(N)$ in which we identify two configurations if they agree on $N\smallsetminus A$. Define the relative section space $\Gamma(N,A)$ to be the space of sections $s$ of $\dot TN|_{N\smallsetminus A} \to N\smallsetminus A$ such that $\{ x\in N\smallsetminus A \;|\; s(x)\neq\infty \}$ is contained in a compact subset of $N$ and such that $s(x)=\infty$ for all $x\in\partial N$. Note that if $N\smallsetminus A$ is a codimension-$0$ submanifold of $N$ which is closed as a subspace, then $\Gamma(N,A)$ is all compactly-supported sections of $\dot{T}(N\smallsetminus A)$ which restrict to $\infty$ on $\partial N\smallsetminus A \subseteq \partial (N\smallsetminus A)$ but not necessarily on $\partial (N\smallsetminus A) \smallsetminus \partial N$. When $A=\varnothing$ we write $C(N,\varnothing)=C(N)$ and $\Gamma(N,\varnothing)=\Gamma(N)$. This agrees with the previous definition of $\Gamma(N)$ when $N$ has empty boundary. There are natural maps $\pi\colon C(N)\to C(N,A)$ and $r\colon \Gamma(N)\to\Gamma(N,A)$, defined by $\pi(c)=[c]$ and $r(\sigma)=\sigma|_{N\smallsetminus A}$.
\end{definition}

Now let $M$ be any connected, smooth, non-compact manifold of dimension $m\geq 2$ and pick an embedding $e\colon M\emb (-1,1)^d\times\bR$ satisfying \eqref{eEmbeddingCondition}. As usual we write $\hat{M} = e(M)$.

\begin{definition}
Define $\mbar = \hat{M}\cup ([-1,1]^{m-1} \times (-3,\infty])$ and $\mbar_t = \mbar \cap \pi_{d+1}^{-1}((-\infty,t))$. Note that $\mathrm{int}(\mbar_t) = \hat{M}_t$. Define $\Gamma(\mbar,e)$ analogously to $\Gamma(M,e)$, in other words,
\[
\Gamma(\mbar,e) = \{ (\sigma,t)\in \Gamma(\mbar)\times [0,\infty) \mid \mathrm{supp}(\sigma)\subseteq \mbar_t \}.
\]
Also let $B_1 = (-1,1)^{m-1} \times [-1,\infty)$ and $B_2 = [-1,1]^{m-1} \times (-2,\infty]$ and define
\begin{align*}
\relC(M,e) &= C(\hat{M},B_1) \\
\relG(M,e) &= \Gamma(\hat{M},\mathring{B}_2) \qquad\text{where }\mathring{B}_2 = \mathrm{int}(B_2) \\
\relG(\mbar,e) &= \Gamma(\mbar,B_2).
\end{align*}
There are maps $C(M,e)\to C(M)$, $\Gamma(M,e)\to \Gamma(M)$ and $\Gamma(\mbar,e)\to \Gamma(\mbar)$ which forget the extra parameter $t$. Composing these with the maps $\pi$ and $r$ from Definition \ref{dRelativeConfigSection} we obtain maps $\pi\colon C(M,e)\to \relC(M,e)$, $r\colon \Gamma(M,e)\to \relG(M,e)$ and $\bar{r}\colon \Gamma(\mbar,e)\to \relG(\mbar,e)$.

There is a map $i\colon \Gamma(M,e) \to \Gamma(\mbar,e)$ given by extending a compactly-supported section of $\hat{M}$ by $\infty$ on $\partial\mbar$. Choosing an isotopy from the identity to an embedding $\mbar\emb\mbar$ taking $\partial\mbar$ into the interior determines a homotopy inverse for $i$. One can similarly define a homotopy equivalence $\reli\colon \relG(M,e) \to \relG(\mbar,e)$, and these commute: $\bar{r}\circ i=\reli\circ r$.

Now also choose a Riemannian metric $g$ on $M$ and a strict lower bound $\rho\colon M\to (0,1)$ for its injectivity radius satisfying conditions \eqref{eCondition1} and \eqref{eCondition2}. In \S\ref{sssStabScanMaps} we defined stabilization and scanning maps depending on this auxiliary data. The stabilization map $T\colon \Gamma(M,e)\to \Gamma(M,e)$ can be extended to a map $T\colon \Gamma(\mbar,e)\to\Gamma(\mbar,e)$ using exactly the same formula as in Definition \ref{dStabilizationMapSection}, so that $T\circ i=i\circ T$.
\end{definition}

\begin{lemma}
There is a well-defined map $\rels\colon \relC(M,e) \to \relG(M,e)$ making the left-hand square below commute on the nose.
\begin{equation}\label{eRelativeScanningSquare}
\centering
\begin{split}
\begin{tikzpicture}
[x=1mm,y=1mm]
\node (tl) at (0,15) {$C(M,e)$};
\node (tm) at (30,15) {$\Gamma(M,e)$};
\node (tr) at (60,15) {$\Gamma(\mbar,e)$};
\node (bl) at (0,0) {$\relC(M,e)$};
\node (bm) at (30,0) {$\relG(M,e)$};
\node (br) at (60,0) {$\relG(\mbar,e)$};
\draw[->] (tl) to node[above,font=\small]{$s$} (tm);
\draw[->] (tm) to node[above,font=\small]{$i$} node[below,font=\small]{$\simeq$} (tr);
\draw[->] (bl) to node[below,font=\small]{$\rels$} (bm);
\draw[->] (bm) to node[below,font=\small]{$\reli$} node[above,font=\small]{$\simeq$} (br);
\draw[->] (tl) to node[left,font=\small]{$\pi$} (bl);
\draw[->] (tm) to node[right,font=\small]{$r$} (bm);
\draw[->] (tr) to node[right,font=\small]{$\bar{r}$} (br);
\end{tikzpicture}
\end{split}
\end{equation}
\end{lemma}

\begin{proof}
Let $(c_i,t_i)\in C(M,e)$ for $i=1,2$, so we have two sections $\sigma_i = r\circ s(c_i,t_i)$ of $\dot{T}(M\smallsetminus B_2)$. Let $p\in M\smallsetminus B_2$. We need to show that $[c_1]=[c_2]$ implies that $\sigma_1(p)=\sigma_2(p)$. Note that, by definition of the scanning map $s$, $\sigma_i(p)$ depends only on $c_i \cap B_{\rho(p)}(p)$, where $\rho\colon M\to (0,1)$ was part of the auxiliary data used to define scanning. But conditions \eqref{eCondition1} and \eqref{eCondition2} imply that $B_{\rho(p)}(p)$ is contained in $M\smallsetminus B_1$ and by assumption, $c_1$ and $c_2$ agree on $M\smallsetminus B_1$.
\end{proof}

Let $\tel(C(M,e))$ be the mapping telescope of the infinite sequence of stabilization maps $t\colon C(M,e) \to C(M,e) \to C(M,e) \to \cdots$. Note that $\pi\circ t = \pi$, so there is a well-defined map $\pi_\tel \colon \tel(C(M,e)) \to \relC(M,e)$. We can similarly use the stabilization maps for $\Gamma(M,e)$ and $\Gamma(\mbar,e)$ to obtain maps $r_\tel \colon \tel(\Gamma(M,e)) \to \relG(M,e)$ and $\bar{r}_\tel \colon \tel(\Gamma(\mbar,e)) \to \relG(\mbar,e)$. The scanning map commutes with the stabilization maps (\S\ref{sssCommutativity}), as does the map $i$, so there are induced maps $s_\tel$ and $i_\tel$ and a commutative square
\begin{equation}\label{eRelativeScanningSquareTel}
\centering
\begin{split}
\begin{tikzpicture}
[x=1mm,y=1mm]
\node (tl) at (0,15) {$\tel(C(M,e))$};
\node (tm) at (35,15) {$\tel(\Gamma(M,e))$};
\node (tr) at (70,15) {$\tel(\Gamma(\mbar,e))$};
\node (bl) at (0,0) {$\relC(M,e)$};
\node (bm) at (35,0) {$\relG(M,e)$};
\node (br) at (70,0) {$\relG(\mbar,e).$};
\draw[->] (tl) to node[above,font=\small]{$s_\tel$} (tm);
\draw[->] (tm) to node[above,font=\small]{$i_\tel$} node[below,font=\small]{$\simeq$} (tr);
\draw[->] (bl) to node[below,font=\small]{$\rels$} (bm);
\draw[->] (bm) to node[below,font=\small]{$\reli$} node[above,font=\small]{$\simeq$} (br);
\draw[->] (tl) to node[left,font=\small]{$\pi_\tel$} (bl);
\draw[->] (tm) to node[right,font=\small]{$r_\tel$} (bm);
\draw[->] (tr) to node[right,font=\small]{$\bar{r}_\tel$} (br);
\end{tikzpicture}
\end{split}
\end{equation}

There are four ingredients for the proof in this section of Theorem \ref{acyclic}: one of these is the special case of Theorem \ref{acyclic} proved in \S\ref{ss-r2crossn} where $(M,e) = (E^m,\iota)$. The other three are the following lemma:

\begin{lemma}\label{lInputs}
In the square \eqref{eRelativeScanningSquareTel},
\begin{itemize2}
\item[\upshape{(a)}] $\pi_\tel$ is an abelian homology fibration,
\item[\upshape{(b)}] $\bar{r}_\tel$ is a quasifibration,
\item[\upshape{(c)}] $\rels$ is a weak equivalence.
\end{itemize2}
\end{lemma}

Before proving these we will show how they imply Theorem \ref{acyclic}.

\begin{proof}[Proof of Theorem \ref{acyclic}]
Recall that $E^m = (-1,1)^{m-1} \times \bR$ and $\iota$ is the inclusion of this into $(-1,1)^d\times \bR$. In \S\ref{ss-r2crossn} we proved Theorem \ref{acyclic} for $(M,e)=(E^m,\iota)$, namely that the scanning map $s_\tel \colon \tel(C(E^m,\iota)) \to \tel(\Gamma(E^m,\iota))$ is acyclic.

Let $\varnothing\in \relC(M,e)$ denote the empty relative configuration and let $\varnothing\in \relG(M,e)$ and $\varnothing\in \relG(\mbar,e)$ denote the infinity sections of $\dot{T}(\hat{M} \smallsetminus \mathring{B}_2)$ and of $\dot{T}(\mbar\smallsetminus B_2)$ respectively. The point-set fiber over $\varnothing\in \relC(M,e)$ in \eqref{eRelativeScanningSquare} is
\begin{align*}
\pi^{-1}(\varnothing) &= \{ (c,t)\in C(\hat{M})\times [0,\infty) \mid c\subseteq B_1 \cap \hat{M}_t \} \quad\subset C(M,e) \\
&\cong \{ (c,t)\in C(\hat{E}^m)\times [0,\infty) \mid c\subseteq B_1 \cap \hat{E}^m_t \} \quad\subset C(E^m,\iota),
\end{align*}
and similarly the point-set fiber $r^{-1}(\varnothing)$ may be thought of as a subspace of $\Gamma(E^m,\iota)$. Moreover, the inclusions $\pi^{-1}(\varnothing) \emb C(E^m,\iota)$ and $r^{-1}(\varnothing) \emb \Gamma(E^m,\iota)$ are homotopy equivalences: in each case one can construct a homotopy inverse using an isotopy $h_t$ of embeddings $\hat{E}^m \emb \hat{E}^m$ from the identity to an embedding $h_1$ with image equal to $(-1,1)^{m-1} \times (-1,\infty)$.

The stabilization map $t\colon C(E^m,\iota) \to C(E^m,\iota)$ restricts to an endomorphism of $\pi^{-1}(\varnothing)$ so we can form the mapping telescope $\tel(\pi^{-1}(\varnothing))$ with respect to this endomorphism and the inclusion $\tel(\pi^{-1}(\varnothing)) \emb \tel(C(E^m,\iota))$ is a homotopy equivalence.\footnote{Since the homotopy inverse to the inclusion $\pi^{-1}(\varnothing) \emb C(E^m,\iota)$ can be constructed to commute with the stabilization map $t$ by ensuring that the isotopy $h_t$ is the identity on $(-1,1)^{m-1} \times [0,\infty)$ at all times $t$.} Similarly, using the stabilization map $T$ for section spaces, we have a mapping telescope $\tel(r^{-1}(\varnothing))$ and an inclusion $\tel(r^{-1}(\varnothing)) \emb \tel(\Gamma(E^m,\iota))$ which is a homotopy equivalence. The following square commutes on the nose:
\begin{center}
\begin{tikzpicture}
[x=1mm,y=1mm]
\node (tl) at (0,15) {$\pi^{-1}(\varnothing)$};
\node (tr) at (40,15) {$C(E^m,\iota)$};
\node (bl) at (0,0) {$r^{-1}(\varnothing)$};
\node (br) at (40,0) {$\Gamma(E^m,\iota).$};
\incl{(tl)}{(tr)}
\incl{(bl)}{(br)}
\draw[->] (tl) to node[left,font=\small]{$s_\varnothing$} (bl);
\draw[->] (tr) to node[right,font=\small]{$s$} (br);
\end{tikzpicture}
\end{center}
A little care is needed here: implicitly we have chosen a Riemannian metric $g_M$ and a function $\rho_M\colon M\to (0,1)$ to define the scanning map $C(M,e)\to \Gamma(M,e)$, which restricts to the map $s_\varnothing$ in this square. We then choose the corresponding data $(g_E,\rho_E)$, which determines $C(E^m,\iota)\to \Gamma(E^m,\iota)$, to agree with $(g_M,\rho_M)$ on $\mathring{B}_2$, which ensures that the square commutes. Hence the square after taking mapping telescopes:
\begin{equation}\label{eComparingFibers1}
\centering
\begin{split}
\begin{tikzpicture}
[x=1mm,y=1mm]
\node (tl) at (0,15) {$\pi_{\tel}^{-1}(\varnothing) = \tel(\pi^{-1}(\varnothing))$};
\node (tr) at (45,15) {$\tel(C(E^m,\iota))$};
\node (bl) at (0,0) {$r_{\tel}^{-1}(\varnothing) = \tel(r^{-1}(\varnothing))$};
\node (br) at (45,0) {$\tel(\Gamma(E^m,\iota))$};
\incl{(tl)}{(tr)}
\incl{(bl)}{(br)}
\draw[->] ($ (tl.south) + (-12,0) $) to node[left,font=\small]{$(s_\tel)_\varnothing$} ($ (bl.north) + (-12,0) $);
\draw[->] (tr) to node[right,font=\small]{$s_\tel$} (br);
\end{tikzpicture}
\end{split}
\end{equation}
also commutes. The map $s_\tel$ is acyclic by \S\ref{ss-r2crossn} and the horizontal maps are homotopy equivalences, so $(s_\tel)_\varnothing$ is also acyclic.

Now consider the squares
\begin{center}
\begin{tikzpicture}
[x=1mm,y=1mm]
\node (tl) at (0,15) {$\tel(\pi^{-1}(\varnothing))$};
\node (tm) at (35,15) {$\tel(r^{-1}(\varnothing))$};
\node (tr) at (70,15) {$\tel(\bar{r}^{-1}(\varnothing))$};
\node (bl) at (0,0) {$\hofib_\varnothing(\pi_\tel)$};
\node (bm) at (35,0) {$\hofib_\varnothing(r_\tel)$};
\node (br) at (70,0) {$\hofib_\varnothing(\bar{r}_\tel)$};
\draw[->] (tl) to node[above,font=\small]{$(s_\tel)_\varnothing$} (tm);
\draw[->] (tm) to node[above,font=\small]{$(i_\tel)_\varnothing$} (tr);
\draw[->] (bl) to node[below,font=\small]{$\hat{s}$} (bm);
\draw[->] (bm) to node[below,font=\small]{$\hat{\imath}$} (br);
\draw[->] (tl) to node[left,font=\small]{$i$} (bl);
\draw[->] (tm) to (bm);
\draw[->] (tr) to node[right,font=\small]{$j$} (br);
\end{tikzpicture}
\end{center}
which compare the set-theoretic and homotopy fibers of the squares \eqref{eRelativeScanningSquareTel}. By Lemma \ref{lInputs}, $i$ is an abelian homology equivalence and $j$ is a weak equivalence, and by above $(s_\tel)_\varnothing$ is acyclic. The map $(i_\tel)_\varnothing$ may be compared to the map $i_\tel$ for $(M,e)=(E^m,\iota)$ similarly to \eqref{eComparingFibers1}, so it is a homotopy equivalence. Hence the bottom horizontal map $\hat{\imath}\circ\hat{s}$ is also an abelian homology equivalence. Moreover, recall from Lemma \ref{lSimple} that $\Gamma(E^m,\iota)$ is a simple space, in particular its fundamental group is abelian. Therefore the same is true for $r^{-1}(\varnothing)$, therefore also for each path-component of $\tel(r^{-1}(\varnothing))$, and therefore also for each path-component of $\hofib_\varnothing(\bar{r}_\tel)$. So all local coefficient systems on $\hofib_\varnothing(\bar{r}_\tel)$ are abelian, so the map $\hat{\imath}\circ\hat{s}$ is in fact acyclic.

Finally, consider the diagram obtained from the outer square of \eqref{eRelativeScanningSquareTel} by taking homotopy fibers horizontally and vertically. By Lemma \ref{lInputs}, $\rels$ is a weak equivalence, and therefore so is the map $\hofib(\hat{\imath}\circ\hat{s}) \to \hofib(i_\tel \circ s_\tel)$ in this diagram. Since acyclicity of a map $f$ is equivalent to $\hofib(f)$ having the integral (untwisted) homology of a point, the fact that $\hat{\imath}\circ\hat{s}$ is acyclic implies that $i_\tel \circ s_\tel$ is also acyclic. But $i_\tel$ is a homotopy equivalence, so $s_\tel$ is acyclic, as required.
\end{proof}

It now remains to prove Lemma \ref{lInputs}. We prove part (a) first, and begin by describing the mapping telescope $\tel(C(M,e))$ a little more.

\begin{definition}
Define $C_\infty(M,e)$ to be the quotient of
\[
\{ (c,t,u)\in C(\hat{M})\times [0,\infty)^2 \mid c\subseteq \hat{M}_t \text{ and } \lvert c\rvert \leq u \leq \lvert c\rvert +1 \}
\]
by the equivalence relation $\sim$ which is the reflexive, symmetric closure of
\[
(c,t,\lvert c\rvert) \sim (c\cup\{p_t\},t+1,\lvert c\rvert +1),
\]
where for $t\in [0,\infty)$, $p_t \coloneqq (\mathbf{0},t+\frac12)\in\hat{M}$. This is homeomorphic to the mapping telescope of the sequence of stabilization maps $C_0(M,e) \to C_1(M,e) \to \cdots$, which was denoted in \S\ref{ss-r2crossn} by $\tel(C_k(M,e))$. There is a descending filtration $C_\infty(M,e) \supset C_{\infty-1}(M,e) \supset \cdots \supset C_{\infty-k}(M,e) \supset \cdots$ defined by
\[
C_{\infty-k}(M,e) = \{ [c,t,u]\in C_\infty(M,e) \mid \lvert c\rvert \geq k \}.
\]
There is a small ambiguity here: to make this unambiguous we specify that the element $[c,t,\lvert c\rvert] = [c\cup\{p_t\},t+1,\lvert c\rvert +1]$ \emph{is} included in $C_{\infty-k}(M,e)$ when $\lvert c\rvert = k-1$. This corresponds to the mapping telescope of the subsequence $C_k(M,e) \to C_{k+1}(M,e) \to C_{k+2}(M,e) \to \cdots$.
\end{definition}

The mapping telescope $\tel(C(M,e))$ is homeomorphic to the disjoint union
\[
\bigsqcup_{k\in\bZ} C_{\infty - \mathrm{max}(k,0)}(M,e),
\]
and we will often talk of the element $[c,t,u]$ of $\tel(C(M,e))$, eliding the $k$ which specifies the path-component. Under this identification, on each path-component $C_{\infty-k}(M,e)$ the map $\pi_\tel$ restricts to the map $C_{\infty-k}(M,e) \to \relC(M,e)$ given by $[c,t,u] \mapsto [c]$.

\begin{lemma}
The map $\pi_\tel$ is an abelian homology fibration.
\end{lemma}

\begin{proof}
Write $Y=\tel(C(M,e))$ and $X=\relC(M,e)$. Define
\[
X_n = \{ [c]\in \relC(M,e) \mid \lvert c\smallsetminus B_1 \rvert \leq n \}.
\]
This is an increasing filtration of $X$ by closed subsets and each compact subset of $X$ is contained in some $X_n$. Also, note that both $X$ and $Y$ are Hausdorff: for $X=\relC(M,e)=C(\hat{M},B_1)$ this uses the fact that $B_1$ is closed in $\hat{M}$ and that $\hat{M}$ is $T_3$ (since it is Hausdorff and locally compact). So we just need to verify the conditions (i)--(iii) of Proposition \ref{strongcrit} in this case.

Since $X_0$ is a point, $\pi_\tel$ is vacuously an abelian homology fibration over it. For all $n\geq 1$ the preimage $\pi_\tel^{-1}(X_n\smallsetminus X_{n-1})$ is the subspace of $Y$ of elements $[c,t,u]$ where the configuration $c$ has exactly $n$ points in $\hat{M} \smallsetminus B_1$. In this subspace no points can pass between $B_1$ and $\hat{M} \smallsetminus B_1$, so $\pi_\tel$ restricted to this subspace is a trivial fiber bundle over $X_n \smallsetminus X_{n-1}$, so in particular an abelian homology fibration over $X_n \smallsetminus X_{n-1}$. This verifies condition (ii). 

For condition (iii), define $U_n$ to be the open subset of $X_n$ consisting of those configurations $c$ with $\lvert c\smallsetminus B_1 \rvert \leq n$ and $\lvert c\smallsetminus \balltwo \rvert \leq n-1$. Let $f_t\colon (-3,0)\to (-3,0)$ be the function:
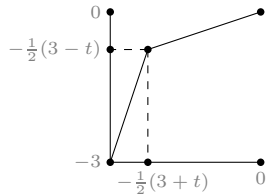
\begin{center}
\begin{tikzpicture}
[x=1mm,y=1mm,scale=2]
\draw[dashed] (2.5,7.5)--(2.5,0);
\draw[dashed] (2.5,7.5)--(0,7.5);
\draw (0,0)--(0,10);
\draw (0,0)--(10,0);
\node at (3.5,0) [anchor=north,font=\scriptsize,black!50]{$-\frac12(3+t)$};
\node at (10,0) [anchor=north,font=\scriptsize,black!50]{$0$};
\node at (0,0) [anchor=east,font=\scriptsize,black!50]{$-3$};
\node at (0,7.5) [anchor=east,font=\scriptsize,black!50]{$-\frac12(3-t)$};
\node at (0,10) [anchor=east,font=\scriptsize,black!50]{$0$};
\foreach \x / \y in {0/0, 2.5/7.5, 10/10, 0/7.5, 0/10, 10/0, 2.5/0} {\node at (\x,\y) [circle,fill,inner sep=1pt]{};}
\draw (0,0)--(2.5,7.5)--(10,10);
\end{tikzpicture}
\end{center}
This induces an automorphism $g_t$ of $\hat{M}$ which is $\mathrm{id}\times f_t$ on $(-1,1)^{m-1}\times (-3,0)$ and the identity elsewhere. We can then define the homotopies $h_t$ and $H_t$ needed for condition (iii) by simply applying $g_t$ to each point of the configuration or relative configuration. Properties (a) and (b) of these homotopies are immediate from their definition.

To check property (c) we will show that for $x\in U_k$ the map $H_1 \colon \pi_\tel^{-1}(x) \to \pi_\tel^{-1}(h_1(x))$ is an abelian homology equivalence. Recall from the proof of Theorem \ref{acyclic} above that there is an inclusion $\pi_{\tel}^{-1}(\varnothing) = \tel(\pi^{-1}(\varnothing)) \hookrightarrow \tel(C(E^m,\iota))$ which is a homotopy equivalence. Also, recall from the proof of Lemma \ref{lSimple} that there is a homotopy-commutative monoid $\cM$ with $\pi_0$ equal to $\bN$ and an element $m_1 \in \cM$ (so that $[m_1]\in\pi_0(\cM)$ is the generator) with the following property. If we write
\[
\tel(\cM) = \tel(\cM \to \cM \to \cdots)
\]
where each map $\cM \to \cM$ is right-multiplication by $m_1$, then there is an inclusion $\tel(\cM) \hookrightarrow \tel(C(E^m,\iota))$ which is a homotopy equivalence. Also note that the fibers of $\pi_\tel$ are all canonically homeomorphic (this of course does not mean that it is locally trivial). Putting this together we have the following (without the dotted arrows):
\begin{equation}\label{eConditionC}
\centering
\begin{split}
\begin{tikzpicture}
[x=1mm,y=1mm]
\node (t1) at (0,15) {$\pi_{\tel}^{-1}(x)$};
\node (t2) at (20,15) {$\pi_{\tel}^{-1}(\varnothing)$};
\node (t3) at (55,15) {$\tel(C(E^m,\iota))$};
\node (t4) at (85,15) {$\tel(\cM)$};
\node (b1) at (0,0) {$\pi_{\tel}^{-1}(h_1(x))$};
\node (b2) at (20,0) {$\pi_{\tel}^{-1}(\varnothing)$};
\node (b3) at (55,0) {$\tel(C(E^m,\iota))$};
\node (b4) at (85,0) {$\tel(\cM)$};
\draw[->] (t1) to node[left,font=\small]{$H_1$} (b1);
\node at (11.5,0) {$\cong$};
\node at (10,15) {$\cong$};
\incl[left]{(t4)}{(t3)}
\incl[left]{(b4)}{(b3)}
\incl{(t2)}{(t3)}
\incl{(b2)}{(b3)}
\draw[->,dashed] (t3) to (b3);
\draw[->,dashed] (t4) to (b4);
\node at (35,-2) [font=\small] {$\simeq$};
\node at (35,17) [font=\small] {$\simeq$};
\node at (73,-2) [font=\small] {$\simeq$};
\node at (73,17) [font=\small] {$\simeq$};
\end{tikzpicture}
\end{split}
\end{equation}

Now $x=[c]\in \relC(M,e)$. Although $[c]$ is a relative configuration, it determines well-defined subconfigurations $c\smallsetminus B_1$ and
\[
g_1(c)\smallsetminus B_1 = c\smallsetminus g_1^{-1}(B_1) = c\smallsetminus \mathrm{cl}(\mathring{B}_2),
\]
so we may define $\ell_x \coloneqq \lvert c\smallsetminus B_1 \rvert - \lvert g_1(c)\smallsetminus B_1 \rvert \in\bN$. This is the number of points of the (relative) configuration $x$ which are pushed into $B_1$ during the homotopy $H_t$. The diagram \eqref{eConditionC} may be completed with dotted arrows so that it homotopy-commutes, and moreover the arrow $\tel(\cM) \to \tel(\cM)$ is the map induced by left-multiplication $(m_1)^{\ell_x}\cdot - \colon \cM \to \cM$.

In \cite[\S 2]{Randal-Williams2013Groupcompletion} it is shown that for a homotopy-commutative monoid $\cM$ and element $m_1\in\cM$ so that $[m_1]$ generates $\pi_0(\cM)=\bN$,\footnote{The fact proved in \cite[\S 2]{Randal-Williams2013Groupcompletion} is more general, but slightly more complicated to state, than what we state here. In particular it does not require that $\pi_0(\cM)=\bN$.} we have the following: for any element $m\in\cM$ the endomorphism $\tel(\cM) \to \tel(\cM)$ induced by left-multiplication $m\cdot -\colon \cM\to\cM$ is an abelian homology equivalence. Applying this to the above situation we see that the right-hand vertical arrow of \eqref{eConditionC} is an abelian homology equivalence, and therefore so is $H_1 \colon \pi_\tel^{-1}(x) \to \pi_\tel^{-1}(h_1(x))$, proving property (c). This completes the verification of condition (iii).

To verify condition (i) we will show that $\pi_\tel$ is locally stalk-like over each $X_n$. There is a basis for $X_n$ consisting of the following open sets. Choose $j\leq n$ small non-overlapping open discs in the manifold $\hat{M}\smallsetminus B_1$ and a real number $\epsilon\in (-3,-1)$ such that no disc intersects $(-1,1)^{m-1}\times (\epsilon,-1)$. Then take the subspace $V$ of $X_n$ of relative configurations where each of the discs contains exactly one point, the strip $(-1,1)^{m-1}\times (\epsilon,-1)$ contains any configuration of at most $n-j$ points and the rest of the manifold $\hat{M}\smallsetminus B_1$ is empty. This subset $V$ is clearly contractible: one can deformation retract it onto the configuration $v_0\in V$ in which the point in each disc is at its center and the strip $D\times (\epsilon,-1)$ is empty. The collection of all such $V$ forms a basis for $X_n$.

We now show that the inclusion $\pi_{\tel}^{-1}(v_0) \hookrightarrow \pi_{\tel}^{-1}(V)$ is a weak equivalence, in fact a homotopy equivalence: we describe a homotopy inverse which lifts the deformation retraction of $V$ onto $\{v_0\}$.\footnote{This is not however a deformation retraction of $\pi_{\tel}^{-1}(V)$ onto $\pi_{\tel}^{-1}(v_0)$ since the subspace $\pi_{\tel}^{-1}(v_0)$ is not fixed during the homotopy.} Observe that $\pi_{\tel}^{-1}(v_0)$ consists of elements $[c,t,u]\in \tel(C(M,e))$ such that
\begin{align*}
c\smallsetminus B_1 =\ &c\smallsetminus ((-1,1)^{m-1} \times [-1,\infty)) \\
\intertext{is precisely the $j$ centers of the discs. On the other hand $\pi_{\tel}^{-1}(V)$ consists of elements $[c,t,u]\in \tel(C(M,e))$ such that}
&c\smallsetminus ((-1,1)^{m-1} \times (\epsilon,\infty))
\end{align*}
consists of one point in each disc. There is a homotopy from the identity on $\pi_{\tel}^{-1}(V)$ to a self-map with image inside $\pi_{\tel}^{-1}(v_0)$ given by gradually pushing each point in a disc towards the center of that disc, and gradually pushing the configuration in $(-1,1)^{m-1} \times (\epsilon,\infty)$ to the right until it lies in $(-1,1)^{m-1} \times (-1,\infty)$.\footnote{Note that it is important that we do not need $\pi_{\tel}^{-1}(v) \hookrightarrow \pi_{\tel}^{-1}(V)$ to be a weak equivalence for \emph{all} $v\in V$, since the previous argument does not work if the configuration $v$ has non-empty intersection with the strip $(-1,1)^{m-1} \times (\epsilon,-1)$.}

This proves that $\pi_{\tel}$ is locally stalk-like over $X_n$, so we have verified condition (i) of Proposition \ref{strongcrit}. Hence $\pi_{\tel}$ is an abelian homology fibration.
\end{proof}

We now turn to part (b) of Lemma \ref{lInputs}, for which we use the following general fact.

\begin{lemma}\label{lRestrictionSerre}
Let $p\colon Y\to X$ be a Serre fibration with a chosen section $s_0\colon X\to Y$. Let $A\subseteq X$ be a relatively compact subset such that $(X,X\smallsetminus A)$ is a relative CW-complex, and let $B\subseteq X\smallsetminus A$ be any subset. Denote by $\Gamma_B(p)$ the space of all compactly-supported sections of $p$ which agree with $s_0$ on $B$ and by $\Gamma_B(p,A)$ the space of all sections of $p^{-1}(X\smallsetminus A) \to X\smallsetminus A$ whose support is contained in a compact subset of $X$ and which agree with $s_0$ on $B$. Then the restriction map $\Gamma_B(p) \to \Gamma_B(p,A)$ is a Serre fibration.
\end{lemma}
\begin{proof}
Any lifting problem for the restriction map induces a lifting problem for $p$, as below:
\begin{center}
\begin{tikzpicture}
[x=1mm,y=1mm,font=\small]
\begin{scope}
\node (tl) at (0,15) {$D^n\times \{0\}$};
\node (tr) at (25,15) {$\Gamma_B(p)$};
\node (bl) at (0,0) {$D^n\times [0,1]$};
\node (br) at (25,0) {$\Gamma_B(p,A)$};
\draw[->] (tl) to (tr);
\draw[->] (bl) to (br);
\incl{(tl)}{(bl)}
\draw[->] (tr) to (br);
\end{scope}
\node at (37,7.5) {$\rightsquigarrow$};
\begin{scope}[xshift=70mm]
\node (tl) at (0,15) {$D^n \times (((X\smallsetminus A)\times [0,1]) \cup (X\times \{0\}))$};
\node (tr) at (40,15) {$Y$};
\node (bl) at (0,0) {$D^n\times (X\times [0,1])$};
\node (br) at (40,0) {$X$};
\draw[->] (tl) to (tr);
\draw[->] (bl) to (br);
\incl{(tl)}{(bl)}
\draw[->] (tr) to node[right]{$p$} (br);
\end{scope}
\end{tikzpicture}
\end{center}
In the right-hand square the left vertical map is the inclusion of a relative CW-complex and $p$ is a Serre fibration, so there is a diagonal map for this square. This gives a map $D^n \times [0,1] \to \Gamma^{\mathit{nc}}_B(p)$, where $\Gamma^{\mathit{nc}}_B(p)$ is the space of \emph{all} (not necessarily compact) sections of $p$ which agree with $s_0$ on $B$. If this map lands in the subspace $\Gamma_B(p)$, we will have found a diagonal for the left-hand square above.

In other words: the diagonal map $D^n \times (X\times [0,1]) \to Y$ restricted to a point in $D^n \times [0,1]$ is a section of $p$, and we need to show that it has compact support. But its restriction to $X\smallsetminus A$ has support contained in a compact subset $K$ of $X$, so the support of the whole section must be contained in $K\cup \bar{A}$, which is compact since $A$ is relatively compact in $X$.
\end{proof}

We also note that a mapping telescope of quasifibrations over a fixed base is again a quasifibration: more precisely for a sequence of spaces $Y_1 \to Y_2 \to \dotsb$ and quasifibrations $f_i\colon Y_i \to X$ commuting with the maps $Y_i\to Y_{i+1}$, the induced map $\mathit{Tel}(f_i)\colon \mathit{Tel}(Y_i)\to X$ is a quasifibration. This is because (a) for any basepoint $x\in X$ we have $\hofib_x(\mathit{Tel}(f_i)) = \mathit{Tel}(\hofib_x(f_i))$ and the subspace $(\mathit{Tel}(f_i))^{-1}(x)$ is the sub-mapping telescope $\mathit{Tel}(f_i^{-1}(x))$, and (b) a telescope of weak equivalences is a weak equivalence.

The map $\bar{r}\colon \Gamma(\mbar,e)\to \relG(\mbar,e)$ factors as
\[
\Gamma(\mbar,e) \longrightarrow \Gamma(\mbar) \longrightarrow \Gamma(\mbar,\mathring{B}_2) \longrightarrow \Gamma(\mbar,B_2)=\relG(\mbar,e).
\]
Applying Lemma \ref{lRestrictionSerre} to the fiber bundle $\dot{T}\mbar \to \mbar$ with the section at infinity, $A=\mathring{B}_2$ and $B=\partial\mbar$ we see that the middle map is a Serre fibration. The first map is a Hurewicz fibration: given a lifting problem $f\colon X\times\{0\} \to \Gamma(\mbar,e)$, $g\colon X\times [0,1] \to \Gamma(\mbar)$ we can define a lift $h\colon X\times [0,1]\to \Gamma(\mbar,e)$ by
\[
h(x,t) = (g(x,t), m(g(x,t)) + \mathrm{pr}_2(f(x,0)) - m(\mathrm{pr}_1(f(x,0))) ),
\]
where $m\colon \Gamma(\mbar)\to [0,\infty)$ takes a configuration to the maximum of the $(d+1)$st coordinates of all its points and $0$. The third map is a homeomorphism, with inverse given by extending a section by $\infty$ on $\partial\mbar \cap B_2$.

Applying the above remark to the sequence of stabilization maps $\Gamma(\mbar,e) \to \Gamma(\mbar,e) \to \cdots$ this implies that the map $\bar{r}_\tel \colon \tel(\Gamma(\mbar,e)) \to \relG(\mbar,e)$ is a quasifibration. Thus we have proved part (b) of Lemma \ref{lInputs}.

As a final input for the proof of Theorem \ref{acyclic} we need to prove part (c) of Lemma \ref{lInputs}, namely that the relative scanning map
\[
\rels \colon \relC(M,e) = C(\hat{M},B_1) \longrightarrow \Gamma(\hat{M},\balltwo) = \relG(M,e)
\]
is a weak equivalence.

\begin{proposition}\label{pRelativeScanning}
The relative scanning map $\rels$ is a weak equivalence.
\end{proposition}

\begin{proof}
This essentially follows from Proposition 2 of \cite{B}. However, no explicit description of the relative scanning map is given there, so we will follow \S 2.3 of \cite{Hesselholt1992} which contains an explicit description and generalizes Proposition 2 of \cite{B} to configurations with labels in a bundle over the manifold (although we will not need this).

Choose an increasing filtration of $\hat{M}$ by compact submanifolds-with-boundary $N_n$. Moreover, choose these so that
\[
N_n \cap \pi_{d+1}^{-1}((-3,\infty)) = [\tfrac{1-n}{n},\tfrac{n-1}{n}]^{m-1} \times (-3,n].
\]
Also define $P_n = B_1 \cap N_n = [\frac{1-n}{n},\frac{n-1}{n}]^{m-1} \times [-1,n]$. Write $s^0$ for the trivial $S^0$-bundle over $N_n$. In general, for a manifold $N$ with boundary $\partial N$, we write $N^{\sfo} = N \cup (\partial N\times [0,1))$. We then have
\begin{align*}
C(N_n,P_n) &\cong C(N_n,P_n;s^0) \\
\Gamma(N_n,P_n) &\cong \Gamma(N_n^{\sfo}\smallsetminus P_n , N_n^{\sfo}\smallsetminus N_n;s^0),
\end{align*}
where the left-hand side is our notation from Definition \ref{dRelativeConfigSection} and the right-hand side is the notation of \cite{Hesselholt1992}. Since $\hat{M}\smallsetminus \mathring{B}_2 \subseteq \hat{M}\smallsetminus B_1$ there is a restriction map $\Gamma(\hat{M},B_1) \to \Gamma(\hat{M},\mathring{B}_2)$, which is clearly a homotopy equivalence. The relative scanning map defined in \cite{Hesselholt1992} (page 198) is a map
\[
\gamma\colon C(N_n,P_n;s^0) \longrightarrow \Gamma(N_n^{\sfo}\smallsetminus P_n , N_n^{\sfo}\smallsetminus N_n;s^0)
\]
which is a weak equivalence by the Proposition in \S 2.3 of \cite{Hesselholt1992}, since $(N_n,P_n)$ is connected. By inspecting the definition of $\gamma$ in \cite{Hesselholt1992} and the definition of $\rels$ at the beginning of this section one sees that the following diagram commutes up to homotopy:
\begin{center}
\begin{tikzpicture}
[x=1mm,y=1mm]
\node (tl) at (0,25) {$C(\hat{M},B_1)$};
\node (tc) at (40,25) {$\Gamma(\hat{M},\mathring{B}_2)$};
\node (tr) at (80,25) {$\Gamma(\hat{M},B_1)$};
\node (ml) at (0,10) {$C(N_n,P_n)$};
\node (mr) at (80,10) {$\Gamma(N_n,P_n)$};
\node (bl) at (0,0) {$C(N_n,P_n;s^0)$};
\node (br) at (80,0) {$\Gamma(N_n^{\sfo}\smallsetminus P_n , N_n^{\sfo}\smallsetminus N_n;s^0)$};
\draw[->] (tl) to node[above,font=\small]{$\rels$} (tc);
\draw[->] (tr) to node[above,font=\small]{$\simeq$} (tc);
\incl{(ml)}{(tl)}
\incl{(mr)}{(tr)}
\draw[->] (bl) to node[above,font=\small]{$\gamma$} (br);
\node at (0,5) {\rotatebox{90}{$=$}};
\node at (80,5) {\rotatebox{90}{$=$}};
\end{tikzpicture}
\end{center}
The spaces $C(N_n,P_n)$ form a filtration of $C(\hat{M},B_1)$ so that every compact subset is contained in some finite stage of the filtration. The map $i\colon \colim_n (C(N_n,P_n))\to C(\hat{M},B_1)$ is therefore a continuous bijection whose inverse is continuous on compact subsets -- hence a weak equivalence. There is an analogous weak equivalence $j$ when $C$ is replaced by $\Gamma$. The map $\gamma$ respects these filtrations so it induces a map $\gamma_\infty\colon \colim_n (C(N_n,P_n))\to \colim_n (\Gamma(N_n,P_n))$ which is a weak equivalence since $\colim(-)$ commutes with $\pi_*(-)$. There is a commutative pentagon formed by $i$, $j$, $\gamma_\infty$ and the top two horizontal maps of the above square. We now deduce that $\rels$ is a weak equivalence since the other four maps in the pentagon are weak equivalences.
\end{proof}

This proves part (c) of Lemma \ref{lInputs}, and therefore completes the proof of Theorem \ref{acyclic}. We can now apply this theorem together with homological stability for oriented configuration spaces to prove Theorem \ref{orientrange} for manifolds admitting boundary.

\begin{corollary}\label{rangeopen}
If $M$ admits boundary the lift $s^+ \colon C^+_k(M,e) \to \Gamma^+_k(M,e)$ of the scanning map induces an isomorphism on $H_*(-;\bZ)$ in the range $* \leq (k-5)/3$ and a surjection for $* \leq (k-2)/3$. 
\end{corollary}

\begin{proof}
We have a commutative square
\begin{center}
\begin{tikzpicture}
[x=1mm,y=1mm]
\node (tl) at (0,15) {$C_k^+(M,e)$};
\node (tr) at (30,15) {$C_{k+1}^+(M,e)$};
\node (bl) at (0,0) {$\Gamma_k^+(M,e)$};
\node (br) at (30,0) {$\Gamma_{k+1}^+(M,e)$};
\draw[->] (tl) to node[above,font=\small]{$t^+$} (tr);
\draw[->] (bl) to node[below,font=\small]{$T^+$} (br);
\draw[->] (tl) to node[left,font=\small]{$s^+$} (bl);
\draw[->] (tr) to node[right,font=\small]{$s^+$} (br);
\end{tikzpicture}
\end{center}
which double covers the square \eqref{eStabScan}. This uses the facts that (a) the oriented configuration space $C_{k+1}^+(M,e)$ pulls back to the oriented configuration space $C_k^+(M,e)$ along the stabilization map $t$ and (b) the double cover $\Gamma_k^+(M,e)$ can be characterized by the property that it pulls back to $C_k^+(M,e)$ along the scanning map $s$. Taking mapping telescopes with respect to $t^+$ and $T^+$ we get a map
\begin{equation}\label{eScanningLimitLifted}
s^+\colon \tel(C_k^+(M,e)) \longrightarrow \tel(\Gamma_k^+(M,e))
\end{equation}
which double-covers the map $s\colon \tel(C_k(M,e)) \to \tel(\Gamma_k(M,e))$. By Theorem \ref{acyclic} the latter is an acyclic map, and therefore so is any homotopy pullback of it, including \eqref{eScanningLimitLifted}. In particular \eqref{eScanningLimitLifted} is an integral homology equivalence.

Now consider the commutative square
\begin{center}
\begin{tikzpicture}
[x=1mm,y=1mm]
\node (tl) at (0,15) {$C_k^+(M,e)$};
\node (tr) at (40,15) {$\tel(C_k^+(M,e))$};
\node (bl) at (0,0) {$\Gamma_k^+(M,e)$};
\node (br) at (40,0) {$\tel(\Gamma_k^+(M,e))$};
\inclusion{above}{$\imath$}{(tl.east)}{(tr.west)}
\inclusion{below}{$\jmath$}{(bl.east)}{(br.west)}
\draw[->] (tl) to node[left,font=\small]{$s^+$} (bl);
\draw[->] (tr) to node[right,font=\small]{$s^+$} (br);
\end{tikzpicture}
\end{center}
where $\imath$ and $\jmath$ are the evident inclusions into the mapping telescopes. Since the stabilization maps $T\colon \Gamma_k(M,e) \to \Gamma_{k+1}(M,e)$ are homotopy equivalences, so is the inclusion $\Gamma_k(M,e) \emb \tel(\Gamma_k(M,e))$. Therefore $\jmath$ is a weak equivalence, since it is a homotopy pullback of this inclusion.

So the bottom and right arrows in the above square are integral homology equivalences -- hence $s^+\colon C_k^+(M,e) \to \Gamma_k^+(M,e)$ is an isomorphism on homology in the same range as $\imath$. Homological stability for oriented configuration spaces \cite{Palmer2013} (Theorem \ref{altstab} in the introduction) implies that $\imath$ is an isomorphism on homology in the claimed range.
\end{proof}

\section{Closed manifolds}\label{ss-closed}

In this section we describe how to extend Corollary \ref{rangeopen} to the case of closed manifolds $M$, which finishes the proof of Theorem \ref{orientrange}. This is based on the arguments used by McDuff to prove Theorem 1.1 of \cite{Mc1}. Since the spaces $C_k^+(M)$ and $\Gamma_k^+(M)$ are both path-connected, the statement of Theorem \ref{orientrange} is true for $k\leq 4$, so we will now assume that $k\geq 5$.

Choose a Riemannian metric on $M$ and an isometric embedding $D\hookrightarrow M$ of the closed, $d$-dimensional unit disc $D$ (we may always scale the metric to make this possible). Following the ideas of \S 6 of \cite{Randal-Williams2013} (see also \S 5 of \cite{Palmer2013}) we define $U_k^+(M)$ to be the subspace of $C_k^+(M)$ of configurations which have a unique closest point in $D$ to the center $0\in D$. The spaces $U_k^+(M)$ and $C_k^+(M\smallsetminus\{0\})$ form an open cover of $C_k^+(M)$ so by excision the homotopy cofiber of the inclusion $C_k^+(M\smallsetminus\{0\}) \hookrightarrow C_k^+(M)$ is homology-equivalent to that of the inclusion $U_k^+(M\smallsetminus\{0\}) \hookrightarrow U_k^+(M)$. For a configuration of $k\geq 3$ points with one marked point, giving an orientation of all $k$ points is equivalent to giving an orientation of the $k-1$ non-marked points. We are assuming that $k\geq 5$, so using this observation we see that the latter inclusion is homeomorphic to the inclusion $(D\smallsetminus \{0\}) \times C_{k-1}^+(M\smallsetminus\{0\}) \hookrightarrow D\times C_{k-1}^+(M\smallsetminus\{0\})$, whose homotopy cofiber is $\Sigma^d(C_{k-1}^+(M\smallsetminus\{0\})_+)$. So we have an identification:
\begin{equation}\label{eHocofib1}
\hocofib(C_k^+(M\smallsetminus\{0\}) \hookrightarrow C_k^+(M)) \;\;\simeq_{H\bZ}\;\; \Sigma^d(C_{k-1}^+(M\smallsetminus\{0\})_+),
\end{equation}
where $(-)_+$ denotes adding an isolated basepoint and $\Sigma^d$ is the $d$th reduced suspension.

We want to similarly identify the homotopy cofiber of the inclusion of section spaces $\Gamma_k^+(M\smallsetminus D) \hookrightarrow \Gamma_k^+(M)$ with $\Sigma^d(\Gamma_k^+(M\smallsetminus D)_+)$, up to homology equivalence. We can choose a CW structure on $M$ so that $(M,D)$ is a relative CW-complex. Also, the subspace $M\smallsetminus D$ is relatively compact (simply because $M$ is compact). Hence, applying Lemma \ref{lRestrictionSerre}, the restriction map $\Gamma(M) \to \Gamma(M,M\smallsetminus D) = \Gamma(D)$ is a Serre fibration. Its fiber $F$ over the section-at-infinity in $\Gamma(D)$ is the space of all sections of $\dot TM\to M$ which are equal to $\infty$ on $D$. This contains a copy of $\Gamma(M\smallsetminus D)$ as a subspace (extending $s\in\Gamma(M\smallsetminus D)$ by $\infty$ on $D$), and the inclusion of this subspace into $F$ is a homotopy equivalence (with homotopy inverse induced by any diffeotopy from the identity $M\to M$ to a diffeomorphism which takes $D$ onto a slightly larger disc).

The restriction of $\Gamma(M)\to \Gamma(D)$ to the path-component $\Gamma_k(M)$ is also a Serre fibration, whose fiber we denote $F_k$. Similarly to the previous paragraph, this contains a homeomorphic copy of $\Gamma_k(M\smallsetminus D)$, and the inclusion $\Gamma_k(M\smallsetminus D) \hookrightarrow F_k$ is a homotopy equivalence.

The composition of $\Gamma_k(M)\to \Gamma(D)$ with the double cover $\Gamma_k^+(M)\to \Gamma_k(M)$ is again a Serre fibration, with fiber $F_k^+$ over the section-at-infinity in $\Gamma(D)$, where $F_k^+ \to F_k$ is the restriction of the double cover $\Gamma_k^+(M) \to \Gamma_k(M)$ to the subspace $F_k \subseteq \Gamma_k(M)$. The further restriction of this double cover to $\Gamma_k(M\smallsetminus D) \subseteq F_k$ is $\Gamma_k^+(M\smallsetminus D)$,\footnote{One can see this from the characterization of $\Gamma_k^+(M)$ in the introduction as the double cover which pulls back to the oriented configuration space $C_k^+(M)$ along the scanning map $C_k(M)\to \Gamma_k(M)$, together with the fact that $C_k^+(M)$ pulls back along the inclusion $C_k(M\smallsetminus D) \hookrightarrow C_k(M)$ to $C_k^+(M\smallsetminus D)$.} so we have an inclusion $\Gamma_k^+(M\smallsetminus D) \hookrightarrow F_k^+$ of double covers over the inclusion $\Gamma_k(M\smallsetminus D)\hookrightarrow F_k$. The latter inclusion is a homotopy equivalence, and therefore the inclusion $\Gamma_k^+(M\smallsetminus D) \hookrightarrow F_k^+$ is a weak equivalence.

Note that $\Gamma(D) \cong \mathrm{Map}(D,S^d) \simeq S^d$. Applying the following general lemma to the Serre fibration $\Gamma_k^+(M)\to \Gamma(D)$ gives us an identification:
\begin{equation}\label{eHocofib2}
\hocofib(\Gamma_k^+(M\smallsetminus D) \hookrightarrow \Gamma_k^+(M)) \;\;\simeq_{H\bZ}\;\; \Sigma^d(\Gamma_k^+(M\smallsetminus D)_+).
\end{equation}

\begin{lemma}\label{lWang}
Let $\pi\colon Y\to X$ be a Serre fibration with $X\simeq S^d$ and let $F = \pi^{-1}(x_0)$ for a point $x_0 \in X$. Then the homotopy cofiber of the inclusion $F\hookrightarrow Y$ is homology equivalent to $\Sigma^d(F_+)$.
\end{lemma}
\begin{proof}
Choose a homotopy equivalence $X\to S^d$ and replace the composition $Y\to X\to S^d$ by a fibration $p\colon Z\to S^d$:
\begin{equation}\label{eFibrations}
\begin{split}
\begin{tikzpicture}
[x=1mm,y=1mm]
\node (tl) at (0,12) {$Y$};
\node (tr) at (20,12) {$Z$};
\node (bl) at (0,0) {$X$};
\node (br) at (20,0) {$S^d$};
\draw[->] (tl) to node[above,font=\small]{$\simeq$} (tr);
\draw[->] (bl) to node[above,font=\small]{$\simeq$} (br);
\draw[->] (tl) to node[left,font=\small]{$\pi$} (bl);
\draw[->] (tr) to node[right,font=\small]{$p$} (br);
\draw[->] (tl) to (br);
\end{tikzpicture}
\end{split}
\end{equation}
Denote the image of $x_0$ in $S^d$ by $y_0$ and denote the fiber $p^{-1}(y_0)$ by $F^\prime$. The homology of the homotopy cofiber of $F\hookrightarrow Y$ is the homology of the pair $(Y,F)$. Using the long exact sequence on homotopy groups for the map of fibrations \eqref{eFibrations} and the long exact sequence on homology for the map of pairs $(Y,F)\to (Z,F^\prime)$ we have $H_*(Y,F)\cong H_*(Z,F^\prime)$. Let $D_+$ and $D_-$ be complementary hemispheres of $S^d$, so $\partial D_+ = \partial D_- = S^{d-1}$, and say $y_0$ is in the interior of $D_-$. The inclusion $F^\prime\hookrightarrow p^{-1}(D_-)$ is a homotopy equivalence, so $H_*(Z,F^\prime) \cong H_*(Z,p^{-1}(D_-))$. By excision, this is isomorphic to $H_*(p^{-1}(D_+),p^{-1}(S^{d-1}))$. Since $D_+$ is contractible, the restriction of the fibration $p$ over $D_+$ is homotopy equivalent (over $D_+$) to the trivial fibration $D_+ \times F^\prime \to D_+$, and so this homology group is in turn isomorphic to $H_*(D_+ \times F^\prime, S^{d-1}\times F^\prime)$. The inclusion $S^{d-1}\times F^\prime \hookrightarrow D_+\times F^\prime$ is a cofibration, so this is the reduced homology of the quotient, which is homeomorphic to $\Sigma^d(F^\prime_+)$. Finally, $F\simeq F^\prime$ so this is $\widetilde{H}_*(\Sigma^d(F_+))$. We therefore have the required isomorphism of homology groups. Moreover, the zig-zag
\begin{center}
\small
\begin{tikzpicture}
[x=1mm,y=1mm]
\node (1) at (0,0) {$(Y,F)$};
\node (2) at (20,0) {$(Z,F^\prime)$};
\node (3) at (45,0) {$(Z,p^{-1}(D_-)$};
\node (4) at (80,0) {$(p^{-1}(D_+),p^{-1}(S^{d-1}))$};
\node (5) at (80,-10) {$(D_+,S^{d-1})\times F^\prime$};
\node (6) at (45,-10) {$(D_+,S^{d-1})\times F$};
\draw[->] (1) to node[above,font=\small]{$\simeq_{H\bZ}$} (2);
\draw[->] (2) to node[above,font=\small]{$\simeq_{\mathrm{w}}$} (3);
\draw[->] (4) to node[above,font=\small]{$\simeq_{H\bZ}$} (3);
\draw[->] (6) to node[above,font=\small]{$\simeq_{\mathrm{w}}$} (5);
\node at (80,-5) {\rotatebox{270}{$\simeq$}};
\end{tikzpicture}
\end{center}
induces a homology equivalence between $(Y,F)$ and $(D^d,S^{d-1})\times F = \Sigma^d(F_+)$.
\end{proof}

With this preliminary work done, we can now easily deduce deduce Theorem \ref{orientrange} for closed manifolds from Theorem \ref{orientrange} for manifolds admitting boundary, which was Corollary \ref{rangeopen}.

\begin{proof}[Proof of Theorem \ref{orientrange} for closed manifolds.]
Let $\hat{D}$ be another closed disc in $M$, containing $D$, which is sufficiently larger than $D$ so that for configurations $c$ in $M\smallsetminus\hat{D}$ the section $s(c)$ is equal to $\infty$ on an open neighborhood of $D$. This means that the scanning map is a map of pairs
\[
s\colon (C_k(M),C_k(M\smallsetminus\hat{D})) \longrightarrow (\Gamma_k(M),\Gamma_k(M\smallsetminus D)),
\]
and it therefore also lifts to a map of pairs
\[
s^+\colon (C_k^+(M),C_k^+(M\smallsetminus\hat{D})) \longrightarrow (\Gamma_k^+(M),\Gamma_k^+(M\smallsetminus D)).
\]

Using the identifications \eqref{eHocofib1} and \eqref{eHocofib2} above, and the fact that the two inclusions $C_k^+(M\smallsetminus \hat{D}) \hookrightarrow C_k^+(M\smallsetminus D) \hookrightarrow C_k^+(M\smallsetminus\{0\})$ are homotopy equivalences, we obtain a map of long exact sequences:
\begin{center}
\begin{tikzpicture}
[x=1mm,y=1mm,font=\footnotesize]
\node (t1) at (0,12) {$H_{i-d+1}(C_{k-1}^+(M\!\smallsetminus\! D))$};
\node (t2) at (35,12) {$\widetilde{H}_i(C_k^+(M\!\smallsetminus\! D))$};
\node (t3) at (65,12) {$\widetilde{H}_i(C_k^+(M))$};
\node (t4) at (95,12) {$H_{i-d}(C_{k-1}^+(M\!\smallsetminus\! D))$};
\node (b1) at (0,0) {$H_{i-d+1}(\Gamma_k^+(M\!\smallsetminus\! D))$};
\node (b2) at (35,0) {$\widetilde{H}_i(\Gamma_k^+(M\!\smallsetminus\! D))$};
\node (b3) at (65,0) {$\widetilde{H}_i(\Gamma_k^+(M))$};
\node (b4) at (95,0) {$H_{i-d}(\Gamma_k^+(M\!\smallsetminus\! D))$};
\draw[->] (t1) to (t2);
\draw[->] (t2) to (t3);
\draw[->] (t3) to (t4);
\draw[->] (b1) to (b2);
\draw[->] (b2) to (b3);
\draw[->] (b3) to (b4);
\draw[->] (t1) to (b1);
\draw[->] (t2) to (b2);
\draw[->] (t3) to (b3);
\draw[->] (t4) to (b4);
\draw[->] (t4) to (113,12);
\draw[->] (b4) to (113,0);
\draw[->] (-20,0) to (b1);
\draw[->] (-20,12) to (t1);
\end{tikzpicture}
\end{center}
The middle two maps are induced by the scanning map $s^+$, and, following through the identifications of the homotopy cofibers above, one sees that the outer vertical maps are induced by the composition $s^+\circ t^+$ of the stabilization map followed by the scanning map.

Since $M\smallsetminus D$ admits boundary $\partial D$, we can apply Theorem \ref{altstab} and Corollary \ref{rangeopen} to conclude that the first, second, fourth and fifth (not drawn) vertical maps above are isomorphisms in the ranges $3i\leq k+3d-9$, $3i\leq k-5$, $3i\leq k+3d-6$ and $3i\leq k-2$ respectively. Since $d\geq 2$ these conditions all hold when $3i\leq k-5$, so by the five-lemma the third vertical map above is an isomorphism in this range. Moreover, in the range $3i\leq k-2$ the second vertical map is a surjection and the fourth and fifth vertical maps are isomorphisms, which implies that the third vertical map is a surjection.
\end{proof}

Using Theorem \ref{orientrange}, we can now prove a similar result regarding the scanning map $C_k(M)\to \Gamma_k(M)$ on twisted homology $H_*(-;\bZ\signtwist)$. Recall that $R\signtwist$, for a ring $R$, is the $R[\bZ/2]$-module where the generator of $\bZ/2$ acts by multiplication by $-1$. The fundamental groups $\pi_1(C_k(M))$ and $\pi_1(\Gamma_k(M))$ have natural maps to $\bZ/2$ described in \S\ref{sssDoubleCover}, so $R\signtwist$ becomes a module over their group-rings too. See the introduction or \cite{BCT} for a discussion of the relationship between $H_*(C_k(M);\bZ\signtwist)$ and the homology of the spaces appearing in the generalized Snaith splitting introduced in \cite{B}.

\begin{corollary}\label{rangeopentwist}
The homomorphism $H_*( C_k(M);\bZ\signtwist) \to H_*(\Gamma_k(M);\bZ\signtwist)$ induced by the scanning map is an isomorphism for $*\leq (k-5)/3$ and a surjection for $*\leq (k-2)/3$. 
\end{corollary}

\begin{proof}
There is a short exact sequence of $\bZ[\bZ/2]$-modules
\begin{equation}\label{eSESofModules}
0\to \bZ \to \bZ[\bZ/2] \to \bZ\signtwist \to 0
\end{equation}
which via the homomorphism $\pi_1(C_k(M)) \to \bZ/2$ can also be viewed as a short exact sequence of $\bZ[\pi_1(C_k(M))]$-modules, so we obtain a long exact sequence of the homology groups of $C_k(M)$ with these coefficients. We can also do the same for $\Gamma_k(M)$ using the homomorphism $\pi_1(\Gamma_k(M)) \to \bZ/2$ instead. Since the homomorphism $\pi_1(C_k(M)) \to \bZ/2$ factors through this one via the scanning map, it induces a map between these long exact sequences. Identifying $H_i(C_k(M);\bZ[\bZ/2])$ with $H_i(C_k^+(M);\bZ)$ and $H_i(\Gamma_k(M);\bZ[\bZ/2])$ with $H_i(\Gamma_k^+(M);\bZ)$, and abbreviating $H_i(-) = H_i(-;\bZ)$, this map of long exact sequences is:
\begin{center}
\begin{tikzpicture}
[x=0.9mm,y=1mm,font=\small]
\node (t1) at (0,12) {$H_i(C_k(M))$};
\node (t2) at (28,12) {$H_i(C_k^+(M))$};
\node (t3) at (60,12) {$H_i(C_k(M);\bZ\signtwist)$};
\node (t4) at (95,12) {$H_{i-1}(C_k(M))$};
\node (t5) at (125,12) {$H_{i-1}(C_k^+(M))$};
\node (b1) at (0,0) {$H_i(\Gamma_k(M))$};
\node (b2) at (28,0) {$H_i(\Gamma_k^+(M))$};
\node (b3) at (60,0) {$H_i(\Gamma_k(M);\bZ\signtwist)$};
\node (b4) at (95,0) {$H_{i-1}(\Gamma_k(M))$};
\node (b5) at (125,0) {$H_{i-1}(\Gamma_k^+(M)).$};
\draw[->] (t1) to (t2);
\draw[->] (t2) to (t3);
\draw[->] (t3) to (t4);
\draw[->] (t4) to (t5);
\draw[->] (b1) to (b2);
\draw[->] (b2) to (b3);
\draw[->] (b3) to (b4);
\draw[->] (b4) to (b5);
\draw[->] (t1) to (b1);
\draw[->] (t2) to (b2);
\draw[->] (t3) to (b3);
\draw[->] (t4) to (b4);
\draw[->] (t5) to (b5);
\end{tikzpicture}
\end{center}
In the range $3i\leq k-5$ the first, second, fourth and fifth vertical maps above are isomorphisms, by Theorems \ref{scanrange} and \ref{orientrange}, so by the five-lemma the middle vertical map is also an isomorphism. In the range $3i\leq k-2$ the second, fourth and fifth vertical maps are isomorphisms, so the middle vertical map is a surjection.
\end{proof}

\small
\phantomsection
\addcontentsline{toc}{section}{References}
\renewcommand{\bibfont}{\normalfont\small}
\defbibnote{myprenote}{}
\printbibliography[prenote=myprenote]

\normalsize

\vspace*{1em}

\noindent Jeremy Miller \hspace*{1em} \href{mailto:jkmiller@math.stanford.edu}{\sffamily jkmiller@math.stanford.edu}

\noindent {\small Department of Mathematics, Stanford University, Building 380, Stanford, California, 94305, USA}

\vspace*{1em}

\noindent Martin Palmer \hspace*{1em} \href{mailto:mpalm_01@uni-muenster.de}{\sffamily mpalm{\textunderscore}01@uni-muenster.de}

\noindent {\small Mathematisches Institut, WWU M{\"u}nster, Einsteinstra{\ss}e 62, 48149 M{\"u}nster, Germany}

\end{document}